\documentclass[a4paper,11pt]{article}
\usepackage[utf8]{inputenc}
\usepackage{amsmath,amsthm}
\usepackage{booktabs}
\usepackage{tikz}
\usepackage{natbib}

\usepackage{dsfont}
\usepackage{amssymb}

\theoremstyle{plain}
\newtheorem{lemma}{Lemma}
\newtheorem{proposition}{Proposition}
\newtheorem{theorem}{Theorem}

\newtheorem{condition}{Condition}

\theoremstyle{definition}
\newtheorem{definition}{Definition}
\newtheorem{example}{Example}

\theoremstyle{remark}
\newtheorem{remark}{Remark}

\DeclareMathOperator{\sgn}{sgn}
\DeclareMathOperator{\divergence}{div}

\DeclareMathOperator{\lspan}{span}
\DeclareMathOperator{\rg}{rg}
\DeclareMathOperator{\TV}{TV}
\DeclareMathOperator{\BV}{BV}
\DeclareMathOperator{\argmin}{argmin}

\newcommand{\conditionalcomma}[1]{\ifx#1\empty\else,\fi}

\newcommand{\RR}{\mathds{R}}
\newcommand{\NN}{\mathds{N}}

\newcommand{\bdry}{\partial}

\newcommand{\grad}{\nabla}
\newcommand{\sign}[1]{\sgn({#1})}
\newcommand{\abs}[1]{{|{#1}|}}
\newcommand{\bigabs}[1]{\bigl|{#1}\bigr|}

\newcommand{\inprod}{\cdot}

\newcommand{\scp}[3][]{\langle{#2},\, {#3}\rangle_{#1}}

\newcommand{\seminorm}[2][]{|{#2}|_{#1}}
\newcommand{\norm}[2][]{\|{#2}\|_{#1}}

\newcommand{\lebesgueL}[1]{L^{#1}}
\newcommand{\lpspace}[1]{\ell^{#1}}
\newcommand{\LPspace}[2]{\lebesgueL{#1}({#2})}

\newcommand{\Hspace}[3][]{H^{#2\conditionalcomma{#1}{#1}}(#3)}

\newcommand{\linspan}[1]{\lspan\{{#1}\}}

\newcommand{\set}[2]{\{{#1} \ : \ {#2}\}}

\newcommand{\sett}[1]{\{{#1}\}}

\newcommand{\without}{\backslash}
\newcommand{\dd}[1]{\ \mathrm{d}{#1}}
\newcommand{\conv}{\ast}
\newcommand{\wrightarrow}{\rightharpoonup}

\newcommand{\seq}[1]{\{{#1}\}}

\newcommand{\subgrad}{\partial}
\newcommand{\placeholder}{\,\cdot\,}
\newcommand{\kernel}[1]{\ker({#1})}
\newcommand{\range}[1]{\rg({#1})}

\newcommand{\directsum}{\oplus}

\newcommand{\kronO}[1]{\mathcal{O}(#1)}

\title{A forward-backward splitting algorithm for
  the minimization of non-smooth convex functionals in 
  Banach space}
\author{Kristian Bredies\thanks{\texttt{kbredies@math.uni-bremen.de},
    Center for Industrial Mathematics
    / Fachbereich 3, University
    of Bremen, Postfach 33 04 40, D-28334 Bremen}}
\date{\today}

\begin{document}
\maketitle

\begin{abstract}
  We consider the task of computing an approximate minimizer 
  of the sum of a smooth and non-smooth
  convex functional, respectively, in Banach space. 
  Motivated by the classical forward-backward 
  splitting method for the subgradients 
  in Hilbert space, we propose a generalization 
  which involves the iterative solution of simpler subproblems.
  Descent and convergence properties of this new algorithm are
  studied. Furthermore, the results are applied to the minimization
  of Tikhonov-functionals associated with linear inverse problems
  and semi-norm penalization in Banach spaces. 
  With the help of Bregman-Taylor-distance
  estimates, rates of convergence for the forward-back\-ward splitting
  procedure are obtained. Examples which demonstrate the 
  applicability are given, in particular, a generalization of
  the iterative soft-thresholding method by Daubechies, Defrise and
  De Mol to Banach spaces as well as total-variation based image
  restoration in higher dimensions are presented.
\end{abstract}

\medskip


\section{Introduction}

The purpose of this paper is, on the one hand, to introduce
an iterative forward-backward splitting procedure 
for the minimization of functionals of type
\begin{equation}
  \label{eq:min_prob}
  \min_{u \in X} \ F(u) + \Phi(u)  
\end{equation}
in Banach spaces and 
to analyze its convergence properties. 
Here, $F$ represents a convex smooth functional 
while the convex $\Phi$ is allowed to be non-smooth.
On the other hand, the application of this algorithm 
to Tikhonov functionals associated with linear inverse problems 
in Banach space is studied. We consider, in particular,
general regularization terms which are only powers of semi-norms
instead of norms.
Moreover, examples which show the range of applicability of the algorithm
as a computational method are presented.

The forward-backward splitting algorithm for minimization 
in Banach space suggested in this work
tries to establish a bridge between the well-known forward-backward
splitting in Hilbert space 
\cite{lions1979splittingalgorithms,chen1997fwdbackwdconvrates,
  combettes2005signalrecovery} 
and minimization algorithms in Banach 
space, which require in general more analysis.
For example, in the situation of Banach spaces, gradient-based 
(or subdifferential-based) methods always have to deal with the 
problem that the gradient is an element of the dual space and can
therefore not directly used as a descent direction. One common
approach of circumventing this difficulty is performing the
step in gradient direction in the dual space and use appropriate
duality mappings to link this procedure to the primal space 
\cite{alber1997minnonsmoothconvexfunc,schuster2008tikbanach}.
Such a procedure applied to~\eqref{eq:min_prob}
can be seen as a full explicit step for $\subgrad(F+\Phi)$ and 
convergence can often be achieved with the help of a-priori 
step-size assumptions.
In contrast to this, forward-backward splitting algorithms also
involve an implicit step by applying a resolvent mapping. The notion
of resolvents can also be generalized to Banach spaces by introducing
duality mappings \cite{kohsaka2004convseqmaxmonotonebanach},
combining both explicit and implicit steps to a 
forward-backward splitting algorithm, however, has not been considered
so far.

The paper can be outlined as follows.
We present, in Section~\ref{sec:algorithm}, 
a generalization of the forward-backward splitting algorithm 
which operates in Banach spaces
and coincides with the usual method in case of Hilbert spaces.
The central step for the proposed method is the successive solution
of problems of type
\begin{equation*}
  \min_{v \in X} \ \frac{\norm[X]{v - u^n}^p}{p} + \tau_n
  \Bigl(
  \scp{F'(u^n)}{v} + \Phi(v) \Bigr)
\end{equation*}
which are in general easier to solve than the original problem.
In particular, we will show in 
Section~\ref{sec:descent_convergence}
that the algorithm stops or the functional values 
converge as $n \rightarrow \infty$ with an asymptotic 
rate of $n^{1-p}$ if $X$ is reflexive 
and $F'$ is locally $(p-1)$-Hölder continuous.
There, we only have to assume that the step-sizes obey some upper and
lower bounds.
Moreover, under certain conditions, strong convergence with 
convergence rates will be proven, in particular that the $q$-convexity
of $\Phi$ implies the rate $n^{(1-p)/q}$.

In Section~\ref{sec:application_tikhonov}, we 
will apply these results to the linear inverse problem
of solving $Ku = f$ with $K: X \rightarrow Y$ which is
regularized with a semi-norm of some potentially smaller space
as penalty term and leading to the
problem of minimizing a non-smooth Tikhonov-functional:
\[
\min_{u \in X} \ \frac{\norm[Y]{Ku-f}^p}{p} + \alpha \seminorm[Z]{u}
\]
The forward-backward splitting algorithm
applied to these type of functional will be discussed. 
It turns out that we have convergence with rate
$n^{(1-p)/q}$ whenever the data space $Y$ is $p$-smooth and 
$q$-convex and the regularizing semi-norm is $q$-convex in a certain
sense.

In Section~\ref{sec:examples_applications}, examples are given
on how to compute the algorithm. Also, basic numerical 
calculations are shown. We consider linear inverse problems with 
sparsity constraints in Banach spaces which leads to a generalization
of the results for the iterative soft-thresholding procedure
in \cite{daubechies2003iteratethresh} to Banach spaces.
Moreover, it is also discussed
how to apply the algorithm to linear inverse problem with 
total-variation regularization in higher dimensions where an embedding
into the Hilbert space $\LPspace{2}{\Omega}$ is impossible.
The article finally concludes with some remarks in 
Section~\ref{sec:conclusions}.

\section{A forward-backward splitting algorithm in Banach space}
\label{sec:algorithm}

Let $X$ be a reflexive Banach space in which the 
functional~\eqref{eq:min_prob} has to be minimized.
Assume that both $F$ and $\Phi$ are proper, 
convex and lower semi-continuous functionals such that 
$F + \Phi$ is coercive. In this setting, we 
require that $F$ represents the ``smooth part'' of the 
functional where $\Phi$ is allowed to be non-smooth.
Specifically, is it assumed that $F$ is differentiable with
derivative which is, on each bounded set,
$(p-1)$ Hölder-continuous for some $1 < p \leq 2$, i.e.
\[
\norm[X^*]{F'(u) - F'(v)} \leq \norm[p-1]{F'} \norm[X]{u - v}^{p-1} 
\]
for $\norm[X]{u}, \norm[X]{v} \leq C$.

We propose the following iterative procedure in order to find
a solution of \eqref{eq:min_prob}.
\begin{enumerate}
\item Start with a $u^0 \in X$ with $\Phi(u^0) < \infty$ and 
  $n = 0$. Estimate, using the coercivity of $F+\Phi$,
  a norm bound $\norm[X]{u} \leq C$ for all $(F+\Phi)(u) 
  \leq (F+\Phi)(u^0)$ and choose $\norm[p-1]{F'}$ 
  accordingly. 
\item Compute the next iterate $u^{n+1}$ as follows. Compute $w^n =
  F'(u^n)$ and determine, for an $\tau_n > 0$ satisfying
  \begin{equation}
    \label{eq:step_size_constraint}
    0 < \underline{\tau} \leq \tau_n
    \qquad , \qquad \tau_n \leq
    \frac{p(1-\delta)}{\norm[p-1]{F'}} \ ,
  \end{equation}
  (with $0 < \delta < 1$) 
  the solutions of the auxiliary minimization problem
  \begin{equation}
    \label{eq:min_prob_aux}
    \min_{v \in X} \ \frac{\norm[X]{v - u^n}^p}{p} + \tau_n
    \Bigl(
    \scp{w^n}{v} + \Phi(v) \Bigr) \ .
  \end{equation}
  Choose, if necessary,
  $u^{n+1}$ as a solution of \eqref{eq:min_prob_aux} which
  minimizes $\norm[X]{v - u^n}$ among all solutions $v$ of
  \eqref{eq:min_prob_aux} to ensure that the fixed points of
  $u^n \mapsto u^{n+1}$ are exactly the minimizers 
  of~\eqref{eq:min_prob}.
\item If $u^{n+1}$ is not optimal, continue with $n := n+1$ and
  repeat with Step 2, after optionally adjusting 
  $\norm[p-1]{F'}$ analogously to Step 1.
\end{enumerate}

\begin{remark}
  This iterative procedure can be seen as a generalization
  of the forward-backward splitting procedure for the subgradients
  (which is in turn
  some kind of generalized gradient projection method 
  \cite{bredies2008itersoftconvlinear}) 
  in case of $X$ and $Y$ being a Hilbert spaces and $p=2$. 
  This reads as, in terms of resolvents 
  (see \cite{ekelandtemam1976convex,brezis1973semigroup,
    showalter1997monotoneoperators}, 
  for example, for introductions to these notions),
  \[
  u^{n+1} = (I + \tau_n \subgrad \Phi)^{-1} (I - \tau_n F')(u^n) 
  \]
  or, equivalently,
  \begin{equation}
    \label{eq:forward_backward_splitting_hilbert}
    u^{n+1} = \argmin_{v \in X} \ 
    \frac{\norm[Y]{v - \bigl(u^n 
        - \tau_n F'(u^n)\bigr)}^2}{2} + \tau_n \Phi(v) \ .
  \end{equation}
  Such an operation does not make sense in Banach spaces, since 
  $F'(u^n) \in X^*$ cannot be subtracted from $u^n \in X$.
  However, \eqref{eq:forward_backward_splitting_hilbert} is equivalent 
  to 
  \[
  u^{n+1} = \argmin_{v \in X} \ 
  \frac{\norm[Y]{v - u^n}^2}{2} + \tau_n \bigl(\scp{F'(u^n)}{v} 
  + \Phi(v) \bigr) \ .
  \]
  which only involves a duality pairing for $F'(u^n)$. Hence, 
  one can replace $\norm[Y]{v - u^n}^2/2$ by $\norm[X]{v - u^n}^p/p$ 
  and $X$, $Y$ with Banach spaces and ends up 
  with~\eqref{eq:min_prob_aux}, which defines a sensible 
  operation (see also Proposition~\ref{prop:aux_prob_well_defined}).
  Denoting $J_p = \subgrad \norm[X]{\placeholder}^p/p$ and
  $P_\tau(u,w) = \bigl(J_p(\placeholder - u) + \tau 
  \subgrad \Phi \bigr)^{-1}(w)$, one can interpret 
  $u^n \mapsto \bigl(u^n, -\tau_nF'(u^n) \bigr)$ as 
  being a generalized forward step while
  $\bigl(u^n, -\tau_n F'(u^n) \bigr) \mapsto P_{\tau_n}\bigl(u^n, 
  -\tau_nF'(u^n)\bigr)$ represents the corresponding generalized
  backward step.
\end{remark}

In the following, we will see that the iteration is indeed
well-defined, i.e.~one can always solve the auxiliary problem
\eqref{eq:min_prob_aux}. Moreover, it is shown that the corresponding
fixed points are exactly the solutions of \eqref{eq:min_prob}.
We start with proving the well-definition.

\begin{proposition}
  \label{prop:aux_prob_well_defined}
  The problem
  \begin{equation}
    \label{eq:min_prob_aux_non_iter}
    \min_{v \in X} \ \frac{\norm[X]{v - u}^p}{p} + \tau \Bigl( \scp{w}{v} +
    \Phi(v) \Bigr)
  \end{equation}
  has a solution for each $u \in X$, $w \in X^*$ and $\tau \geq
  0$. Moreover, there always exists a solution $u^{n+1}$ which
  minimizes $\norm[X]{v - u^n}$ among all solutions $v$.
\end{proposition}
\begin{proof}
  The case $\tau = 0$ is trivial, so let $\tau > 0$ in the following.
  We first show the coercivity of the objective functional.
  For this purpose, note that $N_p(v) = \norm[X]{v - u}^p/p$ grows
  faster than $\norm[X]{v}$,
  i.e.~$N_p(v)/\norm[X]{v} \rightarrow \infty$ whenever $\norm[X]{v}
  \rightarrow \infty$. 
  Moreover, for an $u^*$ minimizing~\eqref{eq:min_prob} we have $\Phi(v)
  \geq \Phi(u^*) + \scp{w^*}{v - u^*}$ where $-w^* = -F'(u^*) \in \subgrad
  \Phi(u^*)$ since $F$ is continuous in $X$ implying that $\subgrad(F +
  \Phi) = F' + \subgrad \Phi$, see \cite{showalter1997monotoneoperators}. 
  Hence, we can estimate
  \begin{multline*}
    \frac{\norm[X]{v - u}^p}{p} + \tau \Bigl( \scp{w}{v} +
    \Phi(v) \Bigr) \\
    \geq \norm[X]{v} \Bigl( \frac{N_p(v)}{\norm[X]{v}} -
    \tau \norm[X^*]{w + w^*} + \tau \frac{\Phi(u^*) - 
      \scp{w^*}{u^*}}{\norm[X]{v}}  
    \Bigr) \geq L \norm[X]{v}
  \end{multline*}
  for some $L > 0$ and all $v$ whose norm is large enough,
  showing that the functional in \eqref{eq:min_prob_aux_non_iter} is
  coercive.

  It follows that the functional in \eqref{eq:min_prob_aux_non_iter} is
  proper, convex, lower semi-con\-tin\-uous and coercive in a reflexive
  Banach space, consequently, at least one solution exists. 
  Finally, denote by $M$ the set of solutions of
  \eqref{eq:min_prob_aux_non_iter}, which is non-empty. Also, $M$ is
  convex and closed since the functional in
  \eqref{eq:min_prob_aux_non_iter} is convex and lower
  semi-continuous, respectively, in the reflexive Banach space $X$. 
  Consequently, by standard arguments from calculus of variations,
  \[
  \min_{v \in X} \ \norm[X]{v - u^n} + I_M(v)
  \]
  admits a solution. Thus, $u^{n+1}$ is well-defined.
\end{proof}

\begin{proposition}
  \label{prop:optimal_fixed_points}
  The solutions $u^*$ of the problem \eqref{eq:min_prob} 
  are exactly the fixed points of the iteration 
  (for each $\tau_n > 0$).
\end{proposition}

\begin{proof}
  Suppose that $u^* = u^n$ is optimal. Then, since $\subgrad(F+
  \Phi) =
  F' + \subgrad \Phi$ we know that $w^* = F'(u^*)$ also
  satisfies $-w^* \in \subgrad \Phi(u^*)$.
  Now, the solutions of \eqref{eq:min_prob_aux} can be
  characterized by each $v \in X$ solving the inclusion relation
  \[
  -\tau_nw^n \in J_p(v - u^n) + \tau_n \subgrad\Phi(v)
  \]
  with $J_p$ being the $p$-duality relation in $X$, i.e.~$J_p =
  \subgrad \frac1p\norm{\placeholder}^p$. 
  Obviously, $v =
  u^*$ is a solution and from the requirement that $u^{n+1}$ is the
  solution which minimizes $\norm[X]{v - u^n}$ among all solutions
  follows $u^{n+1} = u^n$, hence
  $u^*$ is a fixed point.
  
  On the other hand, if $u^n= u^{n+1}$, then, for the corresponding
  $w^n$ holds
  \begin{align*}
    -\tau_n w^n &\in J_p(u^{n+1} - u^n) + \tau_n \subgrad \Phi(u^{n+1}) \\
    \Rightarrow \qquad - \tau_n w^n &\in \tau_n \subgrad \Phi(u^n)
  \end{align*}
  since $J_p(0) = \sett{0}$ (by Asplund's theorem, see 
  \cite{cioranescu1990banach}), 
  meaning that $u^n$ is optimal 
  by $w^n = F'(u^n)$.
\end{proof}

These results justify in a way that the proposed iteration, which can
be interpreted as a  fixed-point iteration, makes
sense. It is, however, not clear whether we can achieve convergence
to a minimizer if the optimality conditions are not satisfied.
For this purpose, the algorithm has to be examined more deeply. 
We will use the descent of the objective functional in
\eqref{eq:min_prob} in order to obtain conditions under which
convergence holds.

\section{Descent properties and convergence}
\label{sec:descent_convergence}

This section deals with descent properties for the proximal
forward-backward splitting algorithm. 
It is shown that each iteration leads to a sufficient descent of the
functional whenever \eqref{eq:step_size_constraint} is satisfied.
Since the minimization problem
\eqref{eq:min_prob} is convex, we can also obtain convergence rates
for the functional, which will be done in the following.
The proof is essentially based on four steps which are in part
inspired by the argumentation in 
\cite{dunn1981gradientprojection,bredies2008harditer,
  bredies2008itersoftconvlinear}. 
After introducing
a functional $D(u^n)$ which measures the descent of the objective
functional, we prove a descent property which, subsequently,
leads to convergence rates for the distance of the functional 
values to the minimum. Finally, under assumptions on
the Bregman (and Bregman-Taylor) distances, convergence rates for
the distance to the minimizer follow.

\begin{lemma}
  For each iterate and each $v \in X$, 
  we have the inequality
  \begin{equation}
    \label{eq:iterate_dual_estimate}
    \frac{\scp{J_p(u^n - u^{n+1})}{v - u^{n+1}}}{\tau_n} \leq \Phi(v) -
    \Phi(u^{n+1}) + \scp{w^n}{v - u^{n+1}} \ . 
  \end{equation}
  Furthermore, with
  \begin{equation}
    \label{eq:iterate_descent_function}
    D(u^n) = \Phi(u^n) - \Phi(u^{n+1}) + \scp{w^n}{u^n - u^{n+1}} 
  \end{equation}
  it holds that
  \begin{gather}
    \label{eq:iterate_norm_estimate}
    \frac{\norm[X]{u^n - u^{n+1}}^p}{\tau_n} \leq D(u^n) \ .
  \end{gather}
\end{lemma}
\begin{proof}
  Since $u^{n+1}$ solves \eqref{eq:min_prob_aux} with data $w^n$ and
  $u^n$, the subgradient relation
  \[
  - J_p(u^{n+1} - u^n) \in \tau_n \bigl( \subgrad \Phi(u^{n+1}) + w^n \bigr)
  \]
  holds and, consequently, by the subgradient inequality,
  \[
  \Phi(u^{n+1}) + \scp{w^n}{u^{n+1}} - \tau_n^{-1}\scp{J_p(u^{n+1} - u^n)}{v -
    u^{n+1}} \leq \Phi(v) + \scp{w^n}{v}
  \]
  for each $v \in X$. Rearranging terms and noting that $J_p(-u) =
  -J_p(u)$ for each $u \in X$ yields \eqref{eq:iterate_dual_estimate}.
  The inequality \eqref{eq:iterate_norm_estimate} then follows by
  letting $v = u^n$ and noting that $\scp{J_p(u^n - u^{n+1})}{u^n
    - u^{n+1}} = \norm[X]{u^n - u^{n+1}}^p$ by definition of the
  duality relation.
\end{proof}

Next, we prove a descent property which will be crucial for the 
convergence analysis. It will make use of the ``descent measure'' 
$D(u^n)$ introduced in~\eqref{eq:iterate_descent_function}.

\begin{proposition}
  \label{prop:descent_est}
  The iteration satisfies
  \begin{equation}
    \label{eq:descent_est}
    (F + \Phi)(u^{n+1}) \leq 
    (F+\Phi)(u^n) -  \Bigl( 1 - \tau_n
    \frac{\norm[p-1]{F'}}{p} \Bigr) D(u^n)
  \end{equation}
  with $D(u^n) \geq 0$ defined by \eqref{eq:iterate_descent_function}.
\end{proposition}

\begin{proof}
  First note that from \eqref{eq:iterate_norm_estimate} follows that
  $D(u^n) \geq 0$ and, together with Proposition
  \ref{prop:optimal_fixed_points}, that $D(u^n) = 0$ if and only if
  $u^n$ is optimal.
  Using the definition of $D(u^n)$ gives
  \begin{equation}
    \label{eq:functional_descent_char}
    (F+\Phi)(u^n) - (F+\Phi)(u^{n+1}) = D(u^n) + \bigl( F(u^n) -
    F(u^{n+1}) - \scp{w^n}{u^n - u^{n+1}} \bigr) \ .
  \end{equation}
  Note that we can write
  \begin{equation}
    \label{eq:taylor_rest}
    F(u^{n+1}) - F(u^n) - \scp{w^n}{u^{n+1} - u^n}
    = \int_0^1 \scp{w(t) - w^n}{u^{n+1} - u^n} \dd{t}
  \end{equation}
  with
  $w(t) = F'\bigl( u^n + t (u^{n+1} - u^n) \bigr)$. 
  We now want to estimate the absolute value of \eqref{eq:taylor_rest}
  in terms of $D(u^n)$:
  \begin{align*}
    \Bigl| \int_0^1 \langle w(t) - w^n &\ , \ u^{n+1} - u^n \rangle
    \dd{t}  \Bigr|
    \leq \int_0^1 \norm[X^*]{w(t) - w^n}\norm[X]{u^{n+1} - u^n}
    \dd{t} \\
    &\leq \int_0^1 \norm[p-1]{F'} \norm[X]{t(u^{n+1} - u^n)}^{p-1}
    \norm[X]{u^{n+1} - u^n} \dd{t} \\
    & = \frac{\norm[p-1]{F'}}{p} \norm[X]{u^{n+1} - u^n}^p \leq
    \frac{\tau_n \norm[p-1]{F'}}{p} D(u^n)
  \end{align*}
  by employing the Hölder-continuity assumption as well as the
  estimate \eqref{eq:iterate_norm_estimate}. The claimed statement
  finally follows from the combination of this with
  \eqref{eq:functional_descent_char}.
\end{proof}

Note that Proposition \ref{prop:descent_est} together with
\eqref{eq:step_size_constraint} yields a guaranteed descent of the
functional $F + \Phi$. Since in this case the boundedness $0 <
\underline{\tau} \leq \tau_n$ is also given, the convergence of the
functional values to the minimum is immediate as we will see in the 
following lemma.
But first, introduce the functional distance
\begin{equation}
  \label{eq:functional_dist}
  r_n = (F+\Phi)(u^n) - \bigl( \min_{u \in X} \ (F+\Phi)(u) 
  \bigr)
\end{equation}
which allows us to write~\eqref{eq:descent_est} as:
\begin{equation}
  \label{eq:descent_est_rest}
  r_n - r_{n+1} \geq \Bigl( 1 - \tau_n \frac{\norm[p-1]{F'}}{p} 
  \Bigr) D(u^n)  \ .
\end{equation}

\begin{lemma}
  \label{lem:descent_iterate_step}
  Assume that $F+\Phi$ is coercive or the sequence $\seq{u^n}$ is 
  bounded. Then, the sequence $\seq{r_n}$ according 
  to~\eqref{eq:functional_dist} 
  satisfies
  \[
  r_n - r_{n+1} \geq c_0 r_n^{p'}
  \]
  for each $n$, with $p'$ the dual exponent $1/p + 1/p' = 1$ and some
  $c_0 > 0$.
\end{lemma}

\begin{proof}
  First note that, according to Proposition \ref{prop:descent_est},
  $\seq{r_n}$ is a non-increasing sequence.
  If $F+ \Phi$ is coercive,
  this immediately means $\norm{u^n - u^*} \leq C_1$ for some
  minimizer $u^* \in X$ and all $n$. The same holds true if $\seq{u^n}$ 
  is already bounded.
  
  Observe that the convexity of $F$
  as well as \eqref{eq:iterate_dual_estimate} gives
  \begin{align*}
    r_n & \leq \Phi(u^n) - \Phi(u^*) + \scp{w^n}{u^n - u^*} \\
    & = \scp{w^n}{u^n - u^{n+1}} + \Phi(u^n) - \Phi(u^*) +
    \scp{w^n}{u^{n+1} - u^*} \\
    &= D(u^n)  + \Phi(u^{n+1}) - \Phi(u^*) + \scp{w^n}{u^{n+1} - u^*}
    \\
    &\leq D(u^n) + \tau_n^{-1}\scp{J_p(u^n - u^{n+1})}{u^{n+1} - u^*} \\
    & \leq D(u^n) + \tau_n^{-1} \norm[X^*]{J_p(u^n - u^{n+1})}
    \norm[X]{u^{n+1} - u^*}
    \ .
  \end{align*}
  Now, since $p > 1$,
  \begin{align*}
    \tau_n^{-1} \norm[X^*]{J_p(u^n - u^{n+1})} &= \tau_n^{-1} \norm[X]{u^n -
      u^{n+1}}^{p-1} \\
    &= \bigl( \tau_n^{-1} \norm[X]{u^n - u^{n+1}}^p
    \bigr)^{1/p'}\tau_n^{-1/p}
  \end{align*}
  where $p'$ is the dual exponent, i.e.~$\frac1p + \frac1{p'} = 1$.
  Further, applying \eqref{eq:iterate_norm_estimate},
  \eqref{eq:descent_est_rest} and taking the step-size constraint
  \eqref{eq:step_size_constraint} into account yields
  \[
  \delta r_n \leq r_n - r_{n+1} +  C_1 (r_n - r_{n+1})^{1/p'}
  (\delta \underline{\tau}^{-1})^{1/p} \ .
  \]
  Note that $r_n - r_{n+1} \leq r_0$ since $\seq{r_n}$ is
  non-increasing. This finally gives
  \begin{gather*}
    \delta r_n \leq \bigl(r_0^{1/p} + C_1 (\delta
    \underline{\tau}^{-1})^{1/p} \bigr)
    (r_n - r_{n+1})^{1/p'} \\
    \Rightarrow \quad
    \Bigl(
    \frac{\delta^{1/p'}\underline{\tau}^{1/p}}{(\delta^{-1}\underline{\tau}
      r_0)^{1/p}
      + C_1} \Bigr)^{p'} r_n^{p'} \leq r_n - r_{n+1}
    \ .
  \end{gather*}
\end{proof}

\begin{proposition}
  \label{prop:functional_descent_rate}
  Under the prerequisites of Lemma~\ref{lem:descent_iterate_step},
  the functional values for $\seq{u^n}$ converge to the
  minimum with rate
  \[
  r_n \leq C n^{1-p}
  \]
  for some $C > 0$.
\end{proposition}

\begin{proof}
  Apply the mean value theorem to get the identity
  \[
  \frac{1}{r_{n+1}^{p'-1}} - \frac{1}{r_n^{p'-1}} = \frac{r_n^{p'-1} -
    r_{n+1}^{p'-1}}{(r_{n+1}r_n)^{p'-1}} = \frac{(p'-1)\vartheta^{p'-2}(r_n
    -r_{n+1})}{(r_{n+1}r_n)^{p'-1}}
  \]
  with $r_{n+1} \leq \vartheta \leq r_n$. Thus, $\vartheta^{p'-2} \geq
  r_{n+1}^{p'-1} r_n^{-1}$ and, by Lemma
  \ref{lem:descent_iterate_step},
  \[
  \frac{1}{r_{n+1}^{p'-1}} - \frac{1}{r_n^{p'-1}}
  \geq \frac{(p'-1)c_0 r_{n+1}^{p'-1}r_n^{p'-1}}{(r_{n+1}r_n)^{p'-1}}
  = (p'-1)c_0 \ .
  \]
  Summing up then yields
  \[
  \frac{1}{r_n^{p'-1}} - \frac{1}{r_0^{p'-1}} = \sum_{i=0}^{n-1} 
  \frac{1}{r_{i+1}^{p'-1}} - \frac{1}{r_i^{p'-1}} \geq n (p'-1) c_0
  \]
  and consequently,
  \[
  r_n^{p'-1} \leq \bigl( r_0^{1-p'} + c_0 (p'-1) n \bigr)^{-1}
  \quad \Rightarrow \quad r_n \leq C n^{1-p}
  \]
  since $1/(1-p') = 1-p$.
\end{proof}

This result immediately gives us weak subsequential convergence to
some minimizer.

\begin{proposition}
  \label{prop:weak_convergence}
  In the situation of Lemma~\ref{lem:descent_iterate_step},
  the sequence $\seq{u^n}$ possesses at least one weak accumulation
  point. Each weak accumulation point is a solution of
  \eqref{eq:min_prob}. In particular, if the minimizer $u^*$ is
  unique, then $u^n \wrightarrow u^*$.
\end{proposition}

\begin{proof}
  Since $r_n \leq n^{1-p}$, the sequence is a minimizing
  sequence, thus, due to the weak lower semi-continuity of $F+\Phi$,
  each weakly convergent subsequence is a minimizer. Moreover,
  it also follows that $\seq{u^n}$ is a bounded sequence 
  in the reflexive Banach space $X$,
  meaning that there is a weakly-convergent subsequence.
  The statement $u^n \wrightarrow u^*$ in case of uniqueness follows
  by the usual subsequence argument.
\end{proof}

To establish strong convergence, one has to examine the functionals
$F$ and $\Phi$ more closely. One approach is to consider the following
Bregman-like distance of $\Phi$ in solution $u^*$ of
\eqref{eq:min_prob}:
\begin{equation}
  \label{eq:bregman_distance}
  R(v) = \Phi(v) - \Phi(u^*) + \scp{F'(u^*)}{v - u^*}
\end{equation}
which is non-negative since $-F'(u^*) \in \subgrad \Phi(u^*)$. Note
that if $\subgrad \Phi(u^*)$ is consisting of one point, $R$ is indeed the
Bregman distance. By optimality of $u^*$ and the subgradient inequality, 
we have for the iterates
$u^n$ that
\[
r_n \geq \Phi(u^n) - \Phi(u^*) + \scp{F'(u^*)}{u^n - u^*} 
= R(u^n) 
\]
hence $R(u^n) \leq Cn^{1-p}$ by Proposition
\ref{prop:functional_descent_rate}. Also note that $R(u) = 0$ for
optimal $u$.
The usual way to achieve convergence is to postulate decay behaviour
for $R$ as the argument approaches $u^*$.

\begin{definition}
  \label{def:convexity_functionals}
  Let $\Phi: X \rightarrow \RR \cup \sett{\infty}$ be proper, convex 
  and lower semi-continuous. The functional $\Phi$ is called 
  \emph{totally convex} in $u^* \in X$, if, for each 
  $w \in \subgrad \Phi(u^*)$ and $\seq{u^n}$ it holds that
  \[
  \Phi(u^n) - \Phi(u^*) - \scp{w}{u^n - u^*} \rightarrow 0
  \quad \Rightarrow \quad 
  \norm[X]{u^n - u^*} \rightarrow 0 \ . 
  \]
  Likewise, $\Phi$ is \emph{convex of power-type $q$} (or $q$-convex) 
  in $u^* \in X$ with a $q \in {[{2,\infty}[}$, if for all 
  $M > 0$ and $w \in \subgrad \Phi(u^*)$ there exists a $c > 0$ such 
  that for all $\norm[X]{u - u^*} \leq M$ we have
  \[
  \Phi(u) - \Phi(u^*) - \scp{w}{u - u^*} \geq c \norm[X]{u - u^*}^q
  \ . 
  \]
\end{definition}

The notion of total convexity of functionals 
is well-known in the literature \cite{butnariu1997localmoduliconvexity,
  butnariu2000totallyconvexoptim}, convexity of power type $q$
is also referred to as \emph{$q$-uniform convexity} 
\cite{aze1995uniformlyconvexsmooth}. The former
term is, however, often used in conjunction with norms of Banach spaces
for which an equivalent definition in terms of the modulus of convexity
(or rotundity) is used \cite{xu1991inequalitybanach}.

Now, if $\Phi$ is \emph{totally convex} in $u^*$, then the sequence 
$\seq{u^n}$ also converges strongly to the minimizer since 
$R(u^n) \rightarrow 0$.
Additionally, the minimizer has to be unique since $\norm{u^{**} -
  u^*} > 0$ and $R(u^{**}) = 0$ would violate the total convexity
property.
The latter considerations prove:

\begin{theorem}
  \label{thm:total_convex_convergence}
  If $\Phi$ is totally convex, then $\seq{u^n}$ converges to the
  unique minimizer $u^*$ in the strong sense.
\end{theorem}

\begin{remark}
  The notion of total convexity is well-known in the study of
  convergence of numerical algorithms and can be established for a
  variety of functionals of interest, for instance, for $\Phi(u) =
  \norm[Y]{u}^r$ if $Y$ is (locally) uniformly
  convex and $r > 1$ \cite{butnariu2000totalconvexity}. 
  If $Y$ is continuously embedded in $X$, then
  $\norm[X]{u^n - u^*} \leq C \norm[Y]{u^n - u^*}$, thus $R(u^n)
  \rightarrow 0$ implies $\norm[Y]{u^n - u^*} \rightarrow 0$ and
  consequently $\norm[X]{u^n - u^*} \rightarrow 0$.
\end{remark}

As one can easily see, the notion of $q$-convexity of $\Phi$ in 
$u^*$ is useful to obtain bounds for the speed of convergence:
Additionally to strong convergence ($q$-convex implies totally convex), 
since $\seq{u^n}$ is
bounded, one can choose a bounded neighborhood in which the sequence
is contained and obtains
\[
\norm[X]{u^n - u^*} \leq \bigl( c^{-1} R(u^n) \bigr)^{1/q} \leq ( c^{-1} 
r_n)^{1/q} \leq c^{-1/q} C^{1/q} n^{(1-p)/q} 
\]
for all $n$ meaning that $u^n \rightarrow u^*$ with asymptotic
rate $n^{(1-p)/q}$. So, we can note:

\begin{theorem}
  \label{thm:q-convex_convergence}
  If $\Phi$ is $q$-convex, then $\seq{u^n}$ converges to the
  unique minimizer $u^*$ with asymptotic rate $n^{(1-p)/q}$.
\end{theorem}

\begin{remark}
  Again, a variety of functionals is $q$-convex, typically for $q
  \geq 2$. The notion translates to norms as follows:
  If $Y$ is a Banach space which is convex of
  power-type $q$ for $q > 1$ (see \cite{lindenstrauss1997banach2} 
  for an introduction to this notion) which is
  continuously embedded in $X$, then $\Phi(u) = \norm[Y]{u}^q$
  is $q$-convex. Consequently, one obtains convergence with rate
  $n^{(1-p)/q}$.

  Note that many known spaces are $q$-convex. For instance,  if $r > 1$,
  each $\LPspace{r}{\Omega}$ is convex of power-type $\max\sett{2,r}$ 
  for arbitrary measure spaces and the respective constants 
  are known, see 
  \cite{hanner1956uniformconvexitylp,meir1984uniformconvexitylp}. 
  The analog applies to Sobolev
  spaces $\Hspace[r]{m}{\Omega}$ associated with arbitrary domains
  $\Omega$: They are also convex of power-type $\max\sett{2,r}$.
  In fact, as a consequence of a theorem of Dvoretsky 
  \cite{dvoretzky1961convexbodiesbanach}, 
  any Banach space can be at most convex of power-type $2$
  \cite{lindenstrauss1997banach2}.
\end{remark}

While the Bregman-distance (or the Bregman-like distance $R$) turns
out to be successful in proving convergence for the minimization
algorithm for many functionals, there are some cases, in which $\Phi$
fails to be $q$-convex or totally convex. In these situations, one can
also take the Taylor-distance of $F$ into account, which is the
remainder of the Taylor-expansion of $F$ up to order $1$:
\begin{equation}
  \label{eq:taylor_distance}
  T(v) = F(v) - F(u^*) - \scp{F'(u^*)}{v -u^*}
\end{equation}
Since $F$ is convex, $T(v) \geq 0$, hence one can indeed speak of some
distance (which is in general not symmetric or satisfying the triangle 
inequality). 
In fact, the Taylor-distance is also a Bregman-distance, 
but here we make the distinction in order to emphasize that
the Taylor- and Bregman-distances are the respective smooth and 
non-smooth concepts for measuring distances by means of 
functionals.

With the introduction of $R$ and $T$, the distance of $(F+\Phi)(v)$
to its minimizer can be expressed as
\begin{equation}
  \label{eq:taylor_bregman_split}
  (F+\Phi)(v) - (F+\Phi)(u^*) = T(v) + R(v) \ , 
\end{equation}
which means that it can be split into a Bregman part (with respect to
$\Phi$) and a Taylor part (with respect to $F$). 
In some situations,
this splitting can be useful, especially for minimizing 
functionals of the Tikhonov-type, as the following 
section shows.

\section{Application to Tikhonov functionals}
\label{sec:application_tikhonov}

Consider the general problem of minimizing a typical 
Tikhonov-type functional for a linear inverse problem,
\begin{equation}
  \label{eq:gen_linear_tikhonov}
  \min_{u \in X} \ \frac{\norm[Y]{Ku - f}^r}{r} + \alpha
  \frac{\seminorm[Z]{u}^s}{s}
\end{equation}
where $r > 1$, $s \geq 1$, $\alpha > 0$,
$X$ is a reflexive Banach space, $Y$ a 
Banach space with $K: X \rightarrow Y$ continuously and
some given data $f \in Y$. Finally, let $\seminorm[Z]{\placeholder}$ 
be a semi-norm of the (not necessarily reflexive) 
Banach space $Z$ which is continuously embedded in $X$.
It becomes a minimization problem 
of type~\eqref{eq:min_prob} with
\begin{equation}
  \label{eq:gen_tikhonov_func_split}
  F(u) = \frac{\norm[Y]{Ku - f}^r}{r} \qquad , 
  \qquad \Phi(u) = \frac{\alpha
    \seminorm[Z]{u}^s}{s}   
\end{equation}
with the usual extension $\Phi(u) = \infty$ 
whenever $u \in X \without Z$.
The aim of this section is to analyze~\eqref{eq:gen_linear_tikhonov} 
with respect to the convergence of the forward-backward splitting 
algorithm applied to~\eqref{eq:gen_tikhonov_func_split}.

First of all, we focus on the minimization problem and specify
which semi-norms we want to allow as regularization functionals. Also,
we need a condition on the linear operator $K$.

\begin{condition}
  \label{cond:feasible_seminorm_linop}
  \mbox{}
  \begin{enumerate}
  \item Let $\seminorm[Z]{\placeholder}$ be a semi-norm on the Banach
    space $Z$ such that $Z = Z_0 \directsum Z_1$ with 
    $Z_0 = \set{z \in Z}{\seminorm[Z]{z} = 0}$  closed in $X$
    and $\seminorm[Z]{\placeholder}$ being equivalent to 
    $\norm[Z]{\placeholder}$ on $Z_1$.
  \item
    Let $K: X \rightarrow Y$ be a
    linear and continuous operator such that 
    $K$, restricted to $Z_0 \subset X$ is continuously 
    invertible.
  \end{enumerate}
\end{condition}

\begin{remark}
  The first condition in Condition~\ref{cond:feasible_seminorm_linop} 
  is satisfied for many 
  semi-norms used in practice. For example, consider
  $X = \LPspace{t}{\Omega}$ and $Z = \Hspace[t]{1}{\Omega}$ with  
  $\seminorm[Z]{z} = \norm[t]{\grad z}$.
  Assume that $\Omega$ is a bounded domain, then 
  $Z_0 = \linspan{\chi_\Omega}$ is closed in $X$ and 
  $Z_1 = \set{z \in \Hspace[t]{1}{\Omega}}{\int_\Omega z \dd{x} = 0}$, 
  for instance, gives $Z = Z_0 \directsum Z_1$. 
  Provided that the
  Poincaré-Wirtinger inequality holds (with constant $C$), 
  we have for $z \in Z_1$
  \[
  \norm[t]{z} + \norm[t]{\grad z} \leq (C+1) \norm[t]{\grad z} 
  \leq (C+1) \bigl( \norm[t]{z} + \norm[t]{\grad z} \bigr)
  \]
  meaning that $\seminorm[Z]{\placeholder}$ and $\norm[Z]{\placeholder}$
  are indeed equivalent on $Z_1$.
  
  Moreover, note that by considering $X/(\kernel{K} \cap Z_0)$ and 
  $Z/(\kernel{K} \cap Z_0)$ instead of $X$ and $Z$, respectively, 
  the second condition in 
  Condition~\ref{cond:feasible_seminorm_linop} is satisfied
  whenever the range of $K$ restricted to $Z_0$ is closed
  in $Y$ (by the open mapping theorem). 
  This is in particular the case when $Z_0$ is 
  finite-dimensional.
\end{remark}

Let us briefly obtain the existence of minimizers under the above 
conditions.

\begin{proposition}
  \label{prop:gen_tikhonov_ex}
  Under the assumption that
  Condition~\ref{cond:feasible_seminorm_linop} is satisfied, 
  the minimization problem~\eqref{eq:gen_linear_tikhonov} 
  possesses at least one solution in $X$.
\end{proposition}

\begin{proof}
  First of all, note that since $Z = Z_0 \directsum Z_1$ with 
  $Z_0$ and $Z_1$ closed, there is, due to the closed-graph theorem, 
  a continuous 
  projection $P: Z \rightarrow Z$ such that
  $\range{P} = Z_1$ and $\kernel{P} = Z_0$. With this projection,
  $\norm[Z]{P\placeholder}$ is equivalent to 
  $\seminorm[Z]{\placeholder}$ (see also 
  \cite{kunisch1993regularizationclosed} for a similar
  situation). 
  
  Verify in the following that $F + \Phi$ is coercive in $X$.
  Suppose that $\Phi(u^n) \leq C_1$ for some sequence
  $\seq{u^n} \subset X$. In particular, $\seq{u^n} \subset Z$ by
  the definition of $\Phi$ and
  $\norm[X]{Pu^n} \leq C_2 \norm[Z]{Pu^n} \leq C_3
  \seminorm[Z]{u^n} \leq C_4$ meaning that $Pu^n$ makes sense and
  is bounded in $X$ whenever $\Phi(u^n)$ is bounded.
  Likewise, examining $F$, we get with $Q = I-P: Z \rightarrow Z_0$ 
  and Taylor expansion that
  \begin{align*}
    F(u) &= \frac{\norm[Y]{Ku - f}^r}{r} 
    \geq \frac1r\bigabs{\norm[Y]{KPu -
        f} - \norm[Y]{KQu}}^r \\
    &\geq \frac{\norm[Y]{KPu - f}^r}{r} - \norm[Y]{KPu -
      f}^{r-1}\norm[Y]{KQu} \\
    & \phantom{\geq \ } + \frac{r-1}{2} \max \
    \sett{\norm[Y]{KPu - f}, \norm[Y]{KQu}}^{r-2} \norm[Y]{KQu}^2 \ .
  \end{align*}
  Now suppose there is a sequence $\seq{u^n}$ where
  $F+\Phi$ is bounded. The claim is that $\norm[Y]{KQu^n}$ is also
  bounded. Assume the opposite, i.e.~that  $\norm{KQu^n} \rightarrow
  \infty$. From the above argumentation we have $\norm[X]{Pu^n}
  \leq C_4$, hence $\norm[Y]{KPu^n - f} \leq C_5$. Consequently, 
  we can assume, without loss of generality, 
  that $\norm[Y]{KQu^n} \geq \norm[Y]{KPu^n
    - f}$. The estimate on $F$ then reads as
  \[
  F(u^n) \geq - C_5^{r-1}\norm[Y]{KQu^n} + (r-1) \norm[Y]{KQu^n}^r
  \]
  implying that $F(u^n) \rightarrow \infty$, which is a contradiction.
  
  Consequently, there holds $\norm[Y]{KQu^n} \leq C_6$. 
  By Condition~\ref{cond:feasible_seminorm_linop}, 
  $K$ is continuously invertible on $Z_0 \subset X$, so finally
  \[
  \norm[X]{u^n} \leq \norm[X]{Qu^n} + \norm[X]{Pu^n}
  \leq C_7 \norm[Y]{KQu^n} + C_4 \leq C_8
  \]
  showing that $F+\Phi$ is coercive.
  
  Consequently, $F+\Phi$ is proper, convex, lower
  semi-continuous and coercive on the reflexive Banach space $X$,
  so at least one minimizer exists.
\end{proof}

Next, we turn to examining the problem more closely and
some properties of the functionals in \eqref{eq:gen_linear_tikhonov}.
Both $F$ and $\Phi$ are convex, proper and lower semi-continuous.
The differentiability of $F$, however, is strongly connected with 
the differentiability of the norm in $Y$, since $K$ is arbitrarily 
smooth. Since the  forward-backward splitting algorithm demands also 
some local Hölder-continuity for $F'$, 
we have to impose conditions. The following notion is known to be
directly connected with the differentiability of the norm and the 
continuity of the derivatives.

\begin{definition}
  A Banach space $Y$ is called \emph{smooth of power-type $p$} 
  (or $p$-smooth) with $p \in {]{1,2}]}$, if there exists a
  constant $C > 0$ such that for all $u,v \in Y$ it holds that
  \[
  \frac{\norm[Y]{v}^{p}}{p} - \frac{\norm[Y]{u}^{p}}{p}
  -\scp{j_{p}(u)}{v - u} \leq C \norm[Y]{v - u}^{p} \ .
  \]
  Here, $j_{p} = \subgrad \norm[Y]{\placeholder}^{p}/p$ 
  denotes the $p$-duality mapping 
  between $Y$ and $Y^*$.
\end{definition}

Now the central result which connects $p$-smoothness with 
differentiability of the norm is the following 
(see \cite{xu1991charinequalities}):

\begin{proposition}
  \label{prop:p_smoothness_norm_diff}
  If $Y$ is smooth of power-type $p$, then 
  $\norm[Y]{\placeholder}^{p}/p$ is continuously differentiable
  with derivative $j_{p}$ which is moreover 
  $(p-1)$ Hölder-continuous.

  Furthermore, for $r \geq p$, the functional 
  $\norm[Y]{\placeholder}^r/r$ is 
  continuously differentiable. Its derivative is given by $j_r$
  which is still $(p - 1)$ Hölder-continuous on each bounded 
  subset of $Y$.
\end{proposition}

Assuming $r \geq p$ and that $Y$ is $p$-smooth gives us, 
together with the usual differentiation rules, that $F$ is
differentiable with derivative
\[
F'(u) = K^*j_r(Ku - f) \ .
\]
Since $K$ is continuous, $j_r$ is Hölder continuous of order $p-1$ on 
bounded sets by Proposition~\ref{prop:p_smoothness_norm_diff}
and keeping in mind that $\norm[p-1]{j_r}$ 
might vary on bounded sets, we can estimate the norm 
$\norm[p-1]{F'}$ on bounded sets as follows:
\[
\norm[p-1]{F'} \leq \norm{K^*} \norm[p-1]{j_r} \norm{K}^{p-1} \leq
\norm[p-1]{j_r} \norm{K}^p \ .
\]
Hence, $F'$ possesses the Hölder-continuity required by the
convergence results for the 
forward-backward splitting algorithm.

As already noted in Theorems~\ref{thm:total_convex_convergence} and
\ref{thm:q-convex_convergence}, convexity of the penalty can lead to
convergence with some rate. Here, we also want to extend the results
to certain convexity of semi-norms which we define as follows:

\begin{definition}
  \label{def:seminorm_convexity}
  Let $\seminorm[Z]{\placeholder}$ be a semi-norm according to
  Condition~\ref{cond:feasible_seminorm_linop}. 
  Then, the functional $\Phi = \alpha \seminorm[Z]{\placeholder}^s/s$ 
  is called \emph{totally convex},
  if, for each fixed $z^* \in Z$, 
  each sequence $\seq{z^n}$ in $Z$ and
  $\zeta \in \subgrad\Phi(z^*)$ there holds
  \[
  \alpha \frac{\seminorm[Z]{z^n}^s}{s} - \alpha 
  \frac{\seminorm[Z]{z^*}^s}{s} 
  - \scp{\zeta}{z^n - z^*}
  \rightarrow 0 
  \quad \Rightarrow \quad \seminorm[Z]{z^n - z^*} \rightarrow 0 \ .
  \]
  Likewise, $\Phi = \frac{\seminorm[Z]{\placeholder}^q}{q}$ 
  is called \emph{convex of power-type $q$} in $z^*$ 
  if, for each $M > 0$ and $\zeta \in \subgrad\Phi(z^*)$ 
  there exists a $c>0$ such that for all $\norm[Z]{z - z^*} \leq M$ 
  the estimate
  \[
  \alpha \frac{\seminorm[Z]{z}^q}{q} - 
  \alpha \frac{\seminorm[Z]{z^*}^q}{q} - \scp{\zeta}{z - z^*}  
  \geq c \seminorm[Z]{z - z^*}^q
  \]
  is satisfied.
\end{definition}

\begin{remark}
  \label{rem:convexity_other_powers}
  Analogously to Proposition~\ref{prop:p_smoothness_norm_diff}, one 
  knows that if a $\seminorm[Z]{\placeholder}^q/q$ is convex
  of power-type $q$ on bounded sets for some $q \geq 2$, then the functional 
  $\Phi = \seminorm[Z]{\placeholder}^s/s$ is also convex
  of power-type $q$ for all $1 < s \leq q$ on bounded sets.
  The same holds true for the convexity of power-type for norms 
  according to Definition~\ref{def:convexity_functionals}.
\end{remark}

With this notion, one is able to give convergence statements and rates for
the algorithm applied to the minimization 
problem~\eqref{eq:gen_linear_tikhonov}.

\begin{theorem}
  \label{thm:gen_tikhonov_convergence}
  If, in the situation of Proposition~\ref{prop:gen_tikhonov_ex},
  the space $Y$ is smooth of power-type $p$ for $p \leq r$ 
  and the functional $\Phi = \alpha\seminorm[Z]{\placeholder}^s/s$ 
  as well as the norm in
  $Y$ is totally convex, then 
  $\seq{u^n}$ converges to the unique minimizer $u^*$.
  If, moreover, $\seminorm[Z]{\placeholder}^q/q$ 
  as well as $Y$ 
  is convex of power-type $q$ for $q \geq \max\ \sett{r,s,2}$ 
  in a minimizer $u^*$, 
  then $u^n \rightarrow u^*$ in $X$ with rate $\kronO{n^{(1-p)/q}}$.
\end{theorem}

\begin{proof}
  First verify that the sequence $\seq{u^n}$ is a minimizing
  sequence for which the associated $r_n$ vanish like $n^{1-p}$
  which means verifying the prerequisites of 
  Proposition~\ref{prop:functional_descent_rate}.
  Both $F$ and $\Phi$ are proper, convex and lower semi-continuous in
  the reflexive Banach space $X$, $F'$ is $(p - 1)$ 
  Hölder-continuous on each bounded
  set by Proposition~\ref{prop:p_smoothness_norm_diff} and 
  we already saw in Proposition~\ref{prop:gen_tikhonov_ex} 
  that $F+ \Phi$ is coercive.
  Hence, Proposition~\ref{prop:functional_descent_rate} is 
  applicable. It remains to show the asserted strong convergence. 
  As already mentioned at the end of 
  Section~\ref{sec:descent_convergence},
  the Bregman-distance is no longer sufficient in order to show 
  convergence, so we have to use Taylor-Bregman-distance
  estimates.

  For that purpose, consider the 
  Bregman-like distance $R$ according to~\eqref{eq:bregman_distance} 
  and Taylor-distance $T$ introduced in~\eqref{eq:taylor_distance} 
  (both with respect to $F$ and $\Phi$ as chosen 
  in~\eqref{eq:gen_tikhonov_func_split} and in the 
  minimizer $u^*$).
  Remember that we can 
  split the functional distance $r_n$ according
  to~\eqref{eq:functional_dist} into Bregman and Taylor 
  parts~\eqref{eq:taylor_bregman_split}, 
  i.e.~$r_n = R(u^n) + T(u^n)$. 
  
  First suppose that $\seminorm[Z]{\placeholder}$ is totally convex, 
  meaning that from
  \[
  r_n \geq R(u^n) = \alpha \frac{\seminorm[Z]{u^n}^s}{s} - 
  \alpha \frac{\seminorm[Z]{u^*}^s}{s} + \scp{F'(u^*)}
  {u^n - u^*}
  \]
  and $r_n \leq C_1 n^{1-p}$ 
  (see Proposition~\ref{prop:functional_descent_rate})
  follows $\seminorm[Z]{u^n - u^*} \rightarrow 0$.
  On the other hand, 
  observe analogously that
  \begin{align*}
    r_n \geq T(u^n) &= 
    \frac{\norm[Y]{K u^n - f}^r}{r} -
    \frac{\norm[Y]{K u^* - f}^r}{r} \\
    \notag
    &\phantom{=} - 
    \scp{j_r(K u^* - f)}{(K u^n - f) - (K u^* - f)}
  \end{align*}
  so the total convexity of $\norm[Y]{\placeholder}$ implies
  $\norm[Y]{K(u^n - u^*)} \rightarrow 0$.
  
  \sloppy Note that since $\seq{u^n}$ is a minimizing
  sequence, we can reuse the arguments from 
  Proposition~\ref{prop:gen_tikhonov_ex} as well as the projections $P$ 
  and $Q$ to obtain that $\seq{u^n} \subset Z$ and
  $\norm[Z]{P(u^n - u^*)} \leq C_3$.
  Thus, we consider $u \in Z$ in the following and are eventually
  able to set $u = u^n - u^*$. 
  It holds that $\range{Q} = Z_0$, 
  hence exploiting again the second part of 
  Condition~\ref{cond:feasible_seminorm_linop} gives
  \[
  \norm[Y]{KQu}^q
  \geq c_1\norm[X]{Q u}^q \ .
  \]
  Using the integral form of the Taylor expansion of
  $\tfrac1q\abs{\placeholder}^q$ up to order $2$ yields
  \begin{align*}
    \frac{\norm[Y]{Ku}^q}{q} &\geq
    \frac1q \bigabs{\norm[Y]{KQ u} - 
      \norm[Y]{-KP u}}^q \\
    &\geq \frac{\norm[Y]{KQ u}^q}{q}
    - \norm[Y]{KQ u}^{q-1} \norm[Y]{KP u} 
    + \frac{q-1}{q} \norm[Y]{KP u}^q \\
    & \geq \frac{\norm[Y]{KQ u}^q}{2q} - 
    \bigl( (2(q-1))^{q-1} - (q-1)\bigr) 
    \frac{\norm[Y]{KP u}^q}{q} 
  \end{align*}
  where, for the latter, the inequality 
  $a^{q-1}b\leq (q')^{-1}\varepsilon^{q'}a^q + q^{-1}\varepsilon^{-q} b^q$ for 
  $a,b \geq 0$ has been utilized.
  
  Together with the estimate on $\norm[Y]{KQ u}^q$ and the continuity
  of $K$, one gets
  \begin{align*}
    \norm[X]{Q u}^q &\leq c_1^{-1} \norm[Y]{KQ u}^q 
    \leq C_5 \bigl( 
    \norm[Y]{KP u}^q + \norm[Y]{K u}^q \bigr) \\
    & \leq C_6 \bigl( 
    C_7 \norm{K}^q  \seminorm[Z]{u}^q + \norm[Y]{Ku}^q \bigr)
  \end{align*}
  and consequently
  \begin{align}
    \label{eq:gen_tikh_conv_est}
    \notag
    \norm[X]{u^n - u^*}^q &\leq 
    2^{q-1} \bigl( \norm[X]{P(u^n - u^*)}^q
    + \norm[X]{Q(u^n - u^*)}^q  \bigr) \\
    & \leq C_8 \bigl( \seminorm[Z]{u^n - u^*}^q 
    + \norm[Y]{K(u^n - u^*)}^q \bigr) \ .
  \end{align}
  Since $R(u^n) \rightarrow 0$ and $T(u^n) \rightarrow 0$ 
  imply $\seminorm[Z]{u^n - u^*} \rightarrow 0$ as well as 
  $\norm[Y]{K(u^n - u^*)} \rightarrow 0$, respectively, 
  we have convergence 
  $u^n \rightarrow u^*$ in $X$.
  
  Regarding the uniqueness, assume that $u^{**}$ is also a minimizer, 
  hence $R(u^{**}) = T(u^{**}) = 0$ and $u^{**} \in Z$. 
  The total convexity then yields 
  $\seminorm[Z]{u^{**} - u^*} = 0 \Rightarrow  
  \norm[Z]{P(u^{**} - u^*)} = 0$ as well as 
  $\norm[Y]{K(u^{**} - u^*)} = 0$. 
  From the latter follows $\norm[Z]{Q(u^{**} - u^*)} = 0$ and 
  consequently the uniqueness statement $\norm[Z]{u^{**} - u^*} = 0$.
 
  In case $\seminorm[Z]{\placeholder}^q/q$ is $q$-convex, we can write,
  having Remark~\ref{rem:convexity_other_powers} in mind,
  \begin{equation}
    \label{eq:gen_tikh_conv_bregman}
    \seminorm[Z]{u^n - u^*}^q \leq c_2^{-1}   
    \Bigl( \alpha \frac{\seminorm[Z]{u^n}^s}{s} - \alpha 
    \frac{\seminorm[Z]{u^*}^s}{s} + \scp{F'(u^*)}
    {u^n - u^*} \Bigr)
    = C_9 R(u^n) \ .
  \end{equation}
  where $c_2$ depends on the $X$-norm bound of the sequence 
  $\seq{u^n}$.
  On the other hand, if $Y$ is $q$-convex,
  the Taylor-distance can be estimated as follows:
  \begin{align}
    \label{eq:gen_tikh_conv_taylor}
    \notag
    T(u^n) &= \frac{\norm[Y]{Ku^n - f}^r}{r} -
    \frac{\norm[Y]{Ku^* - f}^r}{r} \\
    \notag
    &\phantom{=} - 
    \scp{j_r(Ku^* - f)}{(Ku^n - f) - (Ku^* - f)} \\
    & \geq c_3 \norm[Y]{K(u^n - u^*)}^q \ .
  \end{align}
  Hence, \eqref{eq:gen_tikh_conv_est} together with
  \eqref{eq:gen_tikh_conv_bregman} and \eqref{eq:gen_tikh_conv_taylor}
  becomes, because 
  of~\eqref{eq:taylor_bregman_split},
  \[
  \norm[X]{u^n - u^*}^q \
  \leq C_8\bigl(C_9R(u^n) + c_3^{-1} T(u^n) \bigr)
  \leq C_{10} r_n \ .
  \]
  But $r_n \leq C_1 n^{1-p}$, hence $u^n \rightarrow u^*$
  with rate $\kronO{n^{(1-p)/q}}$.
\end{proof}

\section{Examples and Applications}
\label{sec:examples_applications}

This section demonstrates some applications
for the iterative minimization procedure discussed 
in this paper.

We start with an example which is of rather general nature. 
It shows that the forward-backward splitting procedure amounts 
to an iterative thresholding-like algorithm when applied to linear
inverse problems with sparsity constraints. Afterwards, numerical
computations showing the performance of this algorithm in the
discrete case are presented.

\begin{example}
  \label{ex:sparsity}
  Consider a linear
  inverse problem with sparsity constraints in Banach spaces.
  The setting can be summarized as follows: We are given a linear
  and continuous operator $K: \lpspace{r} \rightarrow Y$ where
  $r > 1$, $Y$ is a $p$-smooth and $q$-convex 
  Banach space with $p \leq r \leq q$, $q \geq 2$ 
  and some data $f \in Y$ and 
  want to solve the problem of minimizing the Tikhonov functional
  \begin{equation}
    \label{eq:ex_sparsity_tikh_func}
    \min_{u \in \lpspace{r}} \ \frac{\norm[Y]{Ku - f}^r}{r} + 
    \sum_{k=1}^\infty \alpha_k \frac{\abs{u_k}^s}{s}
  \end{equation}
  where $\seq{\alpha_k}$ is a sequence bounded by $0 < 
  \underline{\alpha} \leq \alpha_k \leq \overline{\alpha} < \infty$ and
  $1 \leq s \leq r$.

  Such situations occur, for example, when one tries to solve 
  $Au = f$ for some linear and continuous 
  $A: X \rightarrow Y$, $X$ being an 
  $\LPspace{r}{\RR^d}$ or a Besov space $B^{\sigma,r}_r(\RR^d)$ which
  is equivalently described using a properly rescaled basis 
  synthesis operator $B: \lpspace{r} \rightarrow X$ with respect
  to appropriate scaling functions/wavelets 
  \cite{meyer1992wavelets}. Utilizing the 
  basis-expansion coefficients for regularization (giving the $s$-th
  power of a norm which is equivalent to some $B^{\sigma',s}_s(\RR^d)$) 
  and denoting $K = AB$ then leads to Tikhonov functionals of the
  type~\eqref{eq:ex_sparsity_tikh_func}, see 
  \cite{daubechies2003iteratethresh} for details.
  
  How does the associated
  forward-backward splitting algorithm look like?
  First of all, we assume that the duality map $j_r$ in $Y$ can be
  computed. This is for example the case for 
  $Y = \LPspace{p^*}{\RR^d}$:
  \[
  j_r(u) = \sign{u} \abs{u}^{p^* - 1} \norm[p^*]{u}^{r - p^*} \ .
  \]
  Note that $\LPspace{p^*}{\RR^d}$ is also smooth of power-type
  $p \leq \min\sett{2,p^*}$ and one can, without greater effort,
  compute estimates for the Hölder-constants $\norm[p - 1]{j_r}$ on
  bounded sets provided that $p \leq r$.
  Consequently, it is reasonable to assume that 
  $F'(u) = K^*j_r(Ku - f)$ is computationally accessible as well as 
  the constants needed for~\eqref{eq:step_size_constraint}.
  
  The main difficulty is therefore computing solutions 
  of~\eqref{eq:min_prob_aux} which reads as, denoting 
  $w^n = K^*j_r(Ku^n - f)$,
  \begin{equation}
    \label{eq:ex_sparsity_aux_prob}
    \min_{v \in \lpspace{r}} \ 
    \sum_{k=1}^\infty  \frac{r\norm[r]{v - u^n}^{p - r}}{p} 
    \frac{\abs{v_k - u^n_k}^r}{r}  
    + \tau_n \Bigl( w^n_k v_k + \alpha_k \frac{\abs{v_k}^s}{s} \Bigr) \ .
  \end{equation}
  It is easy to see that the only coupling between different $k$
  is via the factor $z = \tfrac{p}{r}\norm[r]{v - u^n}^{r - p}$ 
  (if $p \neq r$), so one can take the minimizers of 
  \[
  \Psi_{y,\sigma,t}(x) = 
  \frac{\abs{x - y}^r}{r} + \sigma x + t \frac{\abs{x}^s}s
  \]
  as ansatz with $\sigma = zs_nw^n_k$ and 
  $t = z\tau_n \alpha_k$. Since $\sigma$ enters only linearly in
  $\Psi_{y,\sigma,t}$, it is convenient to write  
  $\argmin_{x\in\RR} \ \Psi_{y,\sigma,t}(x) = 
  (\subgrad \Psi_{y,t})^{-1}(-\sigma)$ with $\Psi_{y,t} 
  = \tfrac1r \abs{x-y}^r + \tau \frac{\abs{x}^s}s$.
  The latter can be expressed as 
  $(\subgrad\Psi_{y,t})^{-1} = S_{y,t}$,
  \begin{equation}
    \label{eq:ex_sparsity_thresholding_like}
    S_{y,t}(x) =
    \begin{cases}
      \bigl(\sign{\placeholder - y} \abs{\placeholder - y}^{r-1} + t 
      \sign{\placeholder} \abs{\placeholder}^{s-1} \bigr)^{-1}(x) &
      \text{for} \ s > 1 \\
      y + \sign{T_{y,t}(x)}\abs{T_{y,t}(x)}^{1/(r-1)}
      & \text{for} \ s = 1
    \end{cases}
  \end{equation}
  and
  \[
  T_{y,t}(x) =
  \begin{cases}
    x + t & \text{for} \ x \leq -\sign{y}\abs{y}^{r-1} - t \\
    x - t & \text{for} \ x \geq -\sign{y}\abs{y}^{r-1} + t \\
    -\sign{y}\abs{y}^{r-1} & \text{else.} 
  \end{cases}
  \]
  \sloppy It is notable that for $s=1$, each $S_{y,t}$ is a 
  thresholding-like function (see Figure \ref{fig:thresholding_like}).
  For $r=2$, one can 
  easily verify the identity
  $S_{y,t}(x) = \sign{x+y} \bigl[ \abs{x+y} - t \bigr]_+$
  meaning that $S_{y,t}$ is the well-known soft-thresholding.
  In order to get a solution for~\eqref{eq:ex_sparsity_aux_prob}, 
  one introduces the operator which applies $S_{y,t}$ pointwise,
  i.e.~$\mathbf{S}_{u,t}(w)_k = S_{u_k,t_k}(w_k)$ such 
  that optimality is achieved if and only if
  \[
  v = \mathbf{S}_{u^n,z\tau_n\alpha}(-z\tau_nw^n) \quad , \quad z = 
  \tfrac{p}{r}\norm[r]{v - u^n}^{r-p} \ .
  \]
  For $r = p$ follows $z = 1$, so one can express the solution 
  explicitly, for the other cases, it is necessary to compute the value
  for $z$ numerically. 

  \begin{figure}
    \centering
    \begin{tabular}{cc}
      \begin{tikzpicture}[xscale=0.5,yscale=1,thick]
        \draw[->] (-3,0) -- (3,0) node[right] {$x$};
        \draw[->] (0,-1) -- (0,2) node[left] {$S_{y,t}(x)$};
        
        \draw[gray] ({1.5*abs(0.75)^0.3},0) -- ({1.5*abs(0.75)^0.3},0.75) 
        -- (0,0.75);
        \draw ({1.5*abs(0.75)^0.3},-0.2) -- ({1.5*abs(0.75)^0.3},0.2);
        \draw (-0.4,0.75) node[left] {$y$} -- (0.4,0.75);
        
        \clip (-3,-1) rectangle (3,3);
        \draw[blue] plot[domain=-2:0] 
        ({(-1)*abs(\x-0.75)^0.5-1.5*abs(\x)^0.3},\x)
        -- plot[domain=0.001:0.751] 
        ({(-1)*abs(\x-0.75)^0.5+1.5*abs(\x)^0.3},\x)
        -- plot[domain=0.749:2] ({abs(\x-0.75)^0.5+1.5*abs(\x)^0.3},\x);
      \end{tikzpicture} &
      \begin{tikzpicture}[xscale=0.5,yscale=1,thick]
        \draw[->] (-3,0) -- (3,0) node[right] {$x$};
        \draw[->] (0,-1) -- (0,2) node[left] {$S_{y,t}(x)$};
        \draw ({-1*(abs(0-0.75)^0.3-1.25)},-0.3) 
        -- ({-1*(abs(0-0.75)^0.3-1.25)},0.2);
        \draw ({-1*(abs(0-0.75)^0.3+1.25)},-0.3) 
        -- ({-1*(abs(0-0.75)^0.3+1.25)},0.2);
        \draw[<->] ({-1*(abs(0-0.75)^0.3-1.25)},-0.2) -- 
        node[fill=white] {$2t$}
        ({-1*(abs(0-0.75)^0.3+1.25)},-0.2);
        
        \draw[gray] (1.25,0) -- (1.25,0.75) -- (0,0.75);
        \draw (1.25,-0.2) -- (1.25,0.2);
        \draw (-0.4,0.75) node[left] {$y$} -- (0.4,0.75);
        
        \clip (-3,-1) rectangle (3,3);
        \draw[blue] plot[domain=-2:0] 
        ({(-1)*abs(\x-0.75)^0.3-1.25},\x)
        -- plot[domain=0.001:0.751] 
        ({(-1)*abs(\x-0.75)^0.3+1.25},\x)
        -- plot[domain=0.749:2] ({abs(\x-0.75)^0.3+1.25},\x);
      \end{tikzpicture} \\
      $y = 0.75, t=1.5$ & $y = 0.75, t = 1.25$  \\
      $r=1.5$, $s=1.3$ & $r=1.3, s=1$
    \end{tabular}
    
    \caption{Illustration of the thresholding-like functions according to 
      \eqref{eq:ex_sparsity_thresholding_like} for some parameters.}
    \label{fig:thresholding_like}
  \end{figure}
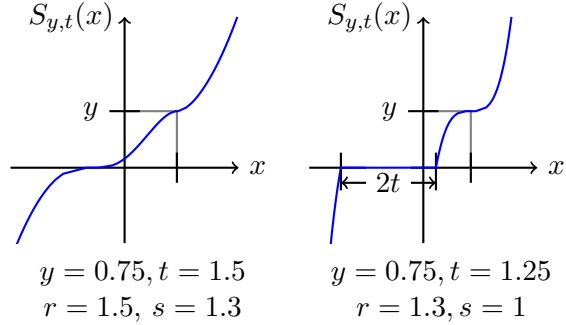

  To introduce some simplification, note that for $u^0 = 0$ follows
  $\norm[r]{u}^s \leq \norm[s]{u}^s \leq \underline{\alpha}^{-1} \Phi(u) 
  \leq (F+\Phi)(u^0)$ yielding $\norm[r]{u} \leq \bigl( \norm[Y]{f}^r/
  (\underline{\alpha}r) \bigr)^{1/s}$. 
  The forward-backward splitting algorithm
  for the solution of~\eqref{eq:ex_sparsity_tikh_func} then
  reads as follows.
  \begin{enumerate}
  \item Initialize with $n=0$, $u^0 = 0$ 
    and choose $\norm[p - 1]{F'}$ locally
    for $\norm[r]{u} \leq  \bigl( \norm[Y]{f}^r/
    (\underline{\alpha}r) \bigr)^{1/s}$.
  \item
    Compute the value 
    \[
    w^n = K^*j_r(Ku^n - f)
    \]
    and solve, for a $\tau_n$ satisfying~\eqref{eq:step_size_constraint},
    the scalar equation
    \[
    u^{n+1} = \mathbf{S}_{u^n,z_n\tau_n\alpha}(-z_n\tau_nw^n) \quad , \quad z_n = 
    \tfrac{p}{r}\norm[r]{u^{n+1} - u^n}^{r - p} \ .
    \]
  \item
    Take, if optimality is not reached, $u^{n+1}$ as the next iterate and
    continue with Step 2 and $n :=  n+1$.
  \end{enumerate}

  It is notable that for $s = 1$, the iterates are always sparse, as
  a consequence of the optimality condition for the auxiliary problem
  \eqref{eq:ex_sparsity_aux_prob} and 
  the fact that sequences in $\lpspace{r'}$ are always null-sequences.
  Regarding the convergence of the algorithm, one easily verifies that
  Theorem~\ref{thm:gen_tikhonov_convergence} is applicable for $s > 1$
  since $\Phi$ is the $s$-th power of a equivalent norm on $\lpspace{r}$
  which is moreover $q$-convex and $Y$ is,
  by assumption, $q$-convex (although this is actually not needed).
  It follows that in that case we have the convergence rate
  $\kronO{n^{(1-p)/q}}$.

  For the case $s=1$, $\Phi$ is not even strictly convex and
  things are a little bit more complicated. But basically, one only
  has to apply the tools already introduced in a slightly different manner.
  Also, the following condition turns out to be crucial for the 
  argumentation: Let us assume that $K$ possesses the 
  \emph{finite basis injectivity} property, that is that for every
  finite index set $J \subset \NN$ the restriction $K|_J$ is injective 
  (meaning that from $Ku =0$ and $u_k = 0$ for each $k \notin J$ follows 
  $u=0$). Under such an assumption, it is also possible to obtain a
  convergence rate of $\kronO{n^{(1-p)/q}}$ where $q$ is the convexity of
  the space $Y$.

  This can be seen analogously to the proof of 
  Theorem~\ref{thm:gen_tikhonov_convergence} and also follows
  a similar line of argumentation in \cite{bredies2008itersoftconvlinear}.
  Observe that for an optimal 
  $u^*$ it holds that $-F'(u^*) = 
  -K^*j_r(Ku^* - f) \in \subgrad \Phi(u^*)$ or, somehow weaker,
  \[
  u^* \ \text{optimal} \quad \Rightarrow \ 
  \begin{cases}
    \abs{K^*j_r(Ku^* - f)}_k \leq \alpha_k & \text{for} \ u_k^* = 0 \\
    \abs{K^*j_r(Ku^* - f)}_k  = \alpha_k & \text{for} \ u_k^* \neq 0 \ .
  \end{cases}
  \]
  Let $u^*$ be a minimizer and
  denote by $J = \set{k \in \NN}{\abs{K^*j_r(Ku^* - f)}_k = \alpha_k}$
  which has to be a finite set since otherwise $K^*j_r(Ku^* - f) 
  \notin \lpspace{r'}$ (remember the assumption $\alpha_k \geq 
  \underline{\alpha} > 0$). 
  Likewise, there exists a $\rho < 1$ such that 
  $\abs{K^*j_r(Ku^* - f)}_k \leq \alpha_k \rho$ for all $k \notin I$.
  For convenience, denote by $P$ the continuous
  projection $(Pu)_k = u_k(\chi_{\NN \without J})_k$, 
  by $Q = I - P$ and define the semi-norms 
  $\seminorm[1]{z} = \norm[1]{Pz}$ in the space $Z = \lpspace{1}$
  as well as $\seminorm[r]{z} = \norm[r]{Pz}$ in $\lpspace{r}$.

  We derive an estimate which somehow states the $q$-convexity
  of $\Phi$ with respect to the semi-norm 
  $\seminorm[r]{\placeholder}$.
  Observe that for $k \notin J$,
  there has to be $u_k^* = 0$, hence one can estimate, for all
  $\norm[r]{v} \leq M$,
  \begin{align*}
    R(v) &\geq \sum_{k \notin J} \alpha_k (\abs{v_k} - \abs{u^*}) + 
    F'(u^*)_k(v_k - u_k^*) \geq \sum_{k \notin J} \alpha_k 
    \bigl( \abs{v_k} - \abs{F'(u^*)_k} \abs{v_k} \bigr) \\
    &\geq (1-\rho) \sum_{k \notin J} \alpha_k  \abs{v_k} 
    \geq (1-\rho) \underline{\alpha} \seminorm[1]{v - u^*}  
    \ .
  \end{align*}
  Then one estimates
  $\seminorm[1]{v - u^*} \geq \seminorm[r]{v - u^*} \geq M^{1-q} 
  \seminorm[r]{v - u^*}^q$ which leads to
  \begin{equation}
    \label{eq:ex_sparsity_bregman_est}
    R(v) \geq c \seminorm[r]{v - u^*}^q 
  \end{equation}
  for a $c > 0$ which only depends on $M$. On the other hand, observe
  that a variant of Condition~\ref{cond:feasible_seminorm_linop} 
  is satisfied: $\lpspace{1} = Z = Z_0 \directsum Z_1$ with $Z_0
  = \range{Q}$
  finite-dimensional and hence closed in $\lpspace{r}$ and 
  $\seminorm[1]{\placeholder}$ being exactly the norm on 
  $Z_1 = \range{P}$.
  By the finite basis injectivity property, $K$ is injective on 
  $\range{Q}$ and since the latter is finite-dimensional, also 
  continuously invertible. But, \eqref{eq:ex_sparsity_bregman_est}
  and the latter is exactly what is utilized in the proof of 
  Theorem~\ref{thm:gen_tikhonov_convergence} to show the desired
  convergence rate. Hence, by repeating the arguments there, one
  obtains $\norm[r]{u^n - u^*} =  \kronO{n^{(1-p)/q}}$.
\end{example}

\begin{figure}
  \centering
  \begin{tikzpicture}[thick,xscale=4.5,yscale=0.06]
    \draw[gray] (-0.025,0) -- (1.025,0);
    \draw[->] (0,-11) -- (1.05,-11) node[right] {$s$};
    \draw[->] (0,-11) -- (0,28);
    \foreach \x / \l in {0/$0$, 0.2/$0.2$, 0.4/$0.4$,
      0.6/$0.6$, 0.8/$0.8$, 1/$1$}
    \draw (\x,-11-1) node[below] {\l} -- (\x,-11+1);
    \foreach \x / \l in {-8/$-8$, 0/$0$, 8/$8$, 16/$16$, 24/$24$}
    \draw (-0.01,\x) node[left] {\l} -- (0.01,\x);
    
    \draw[blue] plot file{progs/u_true.table} node[above] {$u$};
  \end{tikzpicture}
  \begin{tikzpicture}[thick,xscale=4.5,yscale=13.5]
    \draw[->] (0,0) -- (1.05,0) node[right] {$s$};
    \draw[->] (0,0) -- (0,0.15);
    \foreach \x / \l in {0/$0$, 0.2/$0.2$, 0.4/$0.4$,
      0.6/$0.6$, 0.8/$0.8$, 1/$1$}
    \draw (\x,-0.005) node[below] {\l} -- (\x,0.005);
    \foreach \x / \l in {0/$0$, 0.04/$0.04$, 0.08/$0.08$, 0.12/$0.12$}
    \draw (-0.01,\x) node[left] {\l} -- (0.01,\x);

    \draw[blue] plot file{progs/f_true.table} node[below] {$f$};
    \draw[red,only marks,mark=*,mark options={xscale=0.003,yscale=0.001}] plot file{progs/f.table} node[above] {$f_{\delta}$};
  \end{tikzpicture}
  
  \begin{tikzpicture}[thick,xscale=4.5,yscale=0.06]
    \draw[gray] (-0.025,0) -- (1.025,0);
    \draw[->] (0,-11) -- (1.05,-11) node[right] {$s$};
    \draw[->] (0,-11) -- (0,28);
    \foreach \x / \l in {0/$0$, 0.2/$0.2$, 0.4/$0.4$,
      0.6/$0.6$, 0.8/$0.8$, 1/$1$}
    \draw (\x,-11-1) node[below] {\l} -- (\x,-11+1);
    \foreach \x / \l in {-8/$-8$, 0/$0$, 8/$8$, 16/$16$, 24/$24$}
    \draw (-0.01,\x) node[left] {\l} -- (0.01,\x);
    
    \draw[red] plot file{progs/u_true.table} node[above] {$u$};
    \draw[blue] plot file{progs/u_p1.5.table} node[below] {$u_{1.5}$};
  \end{tikzpicture}
  \begin{tikzpicture}[thick,xscale=4.5,yscale=13.5]
    \draw[->] (0,0) -- (1.05,0) node[right] {$s$};
    \draw[->] (0,0) -- (0,0.15);
    \foreach \x / \l in {0/$0$, 0.2/$0.2$, 0.4/$0.4$,
      0.6/$0.6$, 0.8/$0.8$, 1/$1$}
    \draw (\x,-0.005) node[below] {\l} -- (\x,0.005);
    \foreach \x / \l in {0/$0$, 0.04/$0.04$, 0.08/$0.08$, 0.12/$0.12$}
    \draw (-0.01,\x) node[left] {\l} -- (0.01,\x);
    
    \draw[red] plot file{progs/f_true.table} node[above] {$f$};
    \draw[blue] plot file{progs/Ku_p1.5.table} node[below] {$Ku_{1.5}$};
  \end{tikzpicture}

  \begin{tikzpicture}[thick,xscale=4.5,yscale=0.06]
    \draw[gray] (-0.025,0) -- (1.025,0);
    \draw[->] (0,-11) -- (1.05,-11) node[right] {$s$};
    \draw[->] (0,-11) -- (0,28);
    \foreach \x / \l in {0/$0$, 0.2/$0.2$, 0.4/$0.4$,
      0.6/$0.6$, 0.8/$0.8$, 1/$1$}
    \draw (\x,-11-1) node[below] {\l} -- (\x,-11+1);
    \foreach \x / \l in {-8/$-8$, 0/$0$, 8/$8$, 16/$16$, 24/$24$}
    \draw (-0.01,\x) node[left] {\l} -- (0.01,\x);
    
    \draw[red] plot file{progs/u_true.table} node[above] {$u$};
    \draw[blue] plot file{progs/u_p2.table} node[below] {$u_2$};
  \end{tikzpicture}
  \begin{tikzpicture}[thick,xscale=4.5,yscale=13.5]
    \draw[->] (0,0) -- (1.05,0) node[right] {$s$};
    \draw[->] (0,0) -- (0,0.15);
    \foreach \x / \l in {0/$0$, 0.2/$0.2$, 0.4/$0.4$,
      0.6/$0.6$, 0.8/$0.8$, 1/$1$}
    \draw (\x,-0.005) node[below] {\l} -- (\x,0.005);
    \foreach \x / \l in {0/$0$, 0.04/$0.04$, 0.08/$0.08$, 0.12/$0.12$}
    \draw (-0.01,\x) node[left] {\l} -- (0.01,\x);
    
    \draw[red] plot file{progs/f_true.table} node[above] {$f$};
    \draw[blue] plot file{progs/Ku_p2.table} node[below] {$Ku_2$};
  \end{tikzpicture}

  \medskip
  \begin{tabular}{cccccc}
    \toprule
    $p$ & $\norm[p]{Ku_p - f_\delta}$ 
    & $\#\text{NZ}$ & $\norm[p]{f - f_\delta}$ & $\alpha$ \\
    \midrule
    $1.5$ & $0.004651$ & $40$ & $0.004565$ & $0.0025$ \\
    $2$ & $0.004669$ & $61$ & $0.005069$ & $0.000022$ \\
    \bottomrule
  \end{tabular}
  \caption{The outcome of the algorithm for the solution of 
    \eqref{eq:ex_numerical_comp}. 
    Top row: 
    The exact solution (left) and the exact data $f$ which is disturbed
    by noise $f_\delta$ (right). 
    Middle rows: The results of the 
    thresholding-like algorithm $u_{1.5}$ and $u_2$ (left) as well as 
    their images under $K$ (right) 
    for $p=1.5$ and $p=2$, respectively.
    Bottom row: The discrepancy (in the $p$-norm), the number of
    non-zero elements in the solution (out of $500$, $u$ has $9$),
    the data error as well as the regularization parameter 
    associated with $u_{1.5}$ and $u_2$, respectively.
  }
  \label{fig:ex_numerics}
\end{figure}

\begin{example}
  \label{ex:integration_sparsity}
  We like to present numerical computations for a variant of 
  the algorithm developed
  in Example~\ref{ex:sparsity}. The problem we consider is 
  inverting the integration operator on $[0,1]$ in 
  $\LPspace{p}{[0,1]} \rightarrow \LPspace{p}{[0,1]}$
  which is, for simplicity, discretized (with a delta-peak basis) 
  and penalized with the discrete
  $\LPspace{1}{[0,1]}$-norm:
  \begin{equation}
    \label{eq:ex_numerical_comp}
    \min_{u \in \LPspace{p}{[0,1]}} \ \frac{\norm[p]{Ku-f}^p}{p} + 
    \alpha \norm[1]{u} \quad, \quad Ku(t) = \int_0^t u(s) \dd{s} \ .
  \end{equation}
  The forward-backward splitting 
  algorithm applied to the discretized problem then is just 
  the iterative thresholding-like procedure computed in 
  Example~\ref{ex:sparsity} restricted to finitely many dimensions.
  Computations for $p=1.5$ and $p=2$ for noisy data have been performed,
  see Figure \ref{fig:ex_numerics}.
  The regularization parameter has been tuned in order to yield 
  approximately the same discrepancy in the respective norms. 
  As one can see, choosing $p$ 
  less than $2$ may favor more sparsity: Compared to $p=1.5$, the solution
  for $p=2$ has approximately $50\%$ more non-zero elements.
 
  Furthermore, as predicted by the theory, 
  the numerical algorithm indeed converged
  with some rate, in practice, however, it turns out that the 
  convergence is somewhat stable on the one hand but
  very slow on the other hand 
  and many iterations are needed to achieve accurate 
  results.
\end{example}

The following example focuses on presenting an application in which
it is natural to consider the Banach-space setting and on showing
that the forward-backward splitting procedure leads to a convergent
algorithm.

\begin{example}
  \label{ex:tv_3d_deconv}
  Consider the problem of restoring an image in higher dimensions
  with a total-variation penalty term:
  \begin{equation}
    \label{eq:ex_tikhonov_tv_penalty}
    \min_{u \in \LPspace{p}{\Omega}} \ 
    \frac{\norm[p]{Ku - f}^p}{p} + 
    \alpha \TV(u)
  \end{equation}
  with $\Omega \subset \RR^d$ being a bounded domain such that
  $\BV(\Omega)$ is compactly embedded in appropriate Lebesgue spaces.
  Here, $K$ denotes a linear and continuous operator mapping 
  $\LPspace{p}{\Omega} \rightarrow \LPspace{p}{\Omega}$
  with $1 < p \leq d/(d-1)$. This covers in particular the convolution with
  kernels which are only Radon measures, i.e.~$Ku = u \conv k$ with $k \in 
  \mathcal{M}(\RR^d)$, such that 
  the usually assumed continuity 
  $\LPspace{p}{\Omega} \rightarrow \LPspace{2}{\Omega}$
  does not necessarily hold \cite{vese2001bvdenoising}, 
  especially for $d \geq 3$.
  Moreover, in general, coercivity in $\LPspace{2}{\Omega}$ fails and
  it is necessary to consider the Banach space setting. 
  From \cite{acar1994anaboundvariation} we know 
  that~\eqref{eq:ex_tikhonov_tv_penalty} indeed admits a solution in
  $\LPspace{p}{\Omega}$ under general conditions.

  In the following, we focus on the problem of restoring a
  blurred three-dimensional image, i.e.~$Ku = u \conv k$ for some
  non-negative point-spread function $k$ with $\norm[1]{k} = 1$ 
  and $d = 3$. Such a task arises, for example, in confocal microscopy 
  \cite{kempen1997comparisonconfocalmicroscopy}.
  In order to apply the algorithm, we have to solve auxiliary problems of
  the type
  \begin{equation}
    \label{eq:tv_3d_deconv_aux}
    \min_{v\in \LPspace{p}{\Omega}} \frac{\norm[p]{v - u}^p}{p} + s 
    \bigl( \scp{w}{v}  + \alpha \TV(v) \bigr)
  \end{equation}
  with $u \in \LPspace{p}{\Omega}$ and $w \in \LPspace{p'}{\Omega}$ which
  resembles the total-variation denoising functional.
  Following \cite{hintermueller2004bvopti,chambolle2004tvprojection}, 
  one approach to 
  solve~\eqref{eq:tv_3d_deconv_aux} is to consider 
  the Fenchel-predual problem which equivalently reads as
  \[
  \min_{\substack{\norm[\infty]{z} \leq s\alpha \\ z \inprod \nu = 0 \ 
      \text{on} \ \bdry \Omega}} \ 
  \frac{\norm[p']{\divergence z - sw}^{p'}}{p'} + \scp{\divergence z}{u} \ .
  \]
  In the discrete setting, with an appropriate
  linear discrete divergence operator, the above becomes a smooth
  minimization problem with convex constraints which can, for example,
  be solved approximately 
  by a gradient projection method \cite{dunn1981gradientprojection} 
  or a semi-smooth Newton method
  \cite{ng2007semismoothnewtontv}.
  Finally, once a solution $z^*$ of the predual problem is known, a solution
  of~\eqref{eq:tv_3d_deconv_aux} can be obtained by the corresponding
  Fenchel optimality conditions which, in this case, lead to the identity
  $v^* = u + j_{p'}(\divergence z^* - sw)$.

  Hence, one can actually perform the forward-backward splitting procedure
  in practice. Regarding its convergence, we can are only able to apply 
  Proposition~\ref{prop:weak_convergence} and get weak convergence
  of $\seq{u^n}$ to 
  the minimizer $u^*$ in $\LPspace{p}{\Omega}$ (note that the minimizer
  has to be unique since $K$ is injective). But additionally, one easily 
  deduces that $\seq{\BV(u^n)}$ is also bounded which gives, 
  by compactness, the strong convergence in case $1 < p < 3/2$.
  Consequently, considering the inverse problem in Banach space
  yields a convergent algorithm for regularization with the 
  total-variation semi-norm. The convergence, however, comes without
  an estimate for its speed.
  
  Based on the arguments presented above, 
  numerical computations have been carried out. You can
  see the outcome of the algorithm for some sample data in 
  the Figures~\ref{fig:tv_3d_deconv_art} and~\ref{fig:tv_3d_deconv_microscopy}.
\end{example}

\begin{figure}
  \centering
  \begin{tabular}{ccc} \\
    \includegraphics[height=3.3cm]{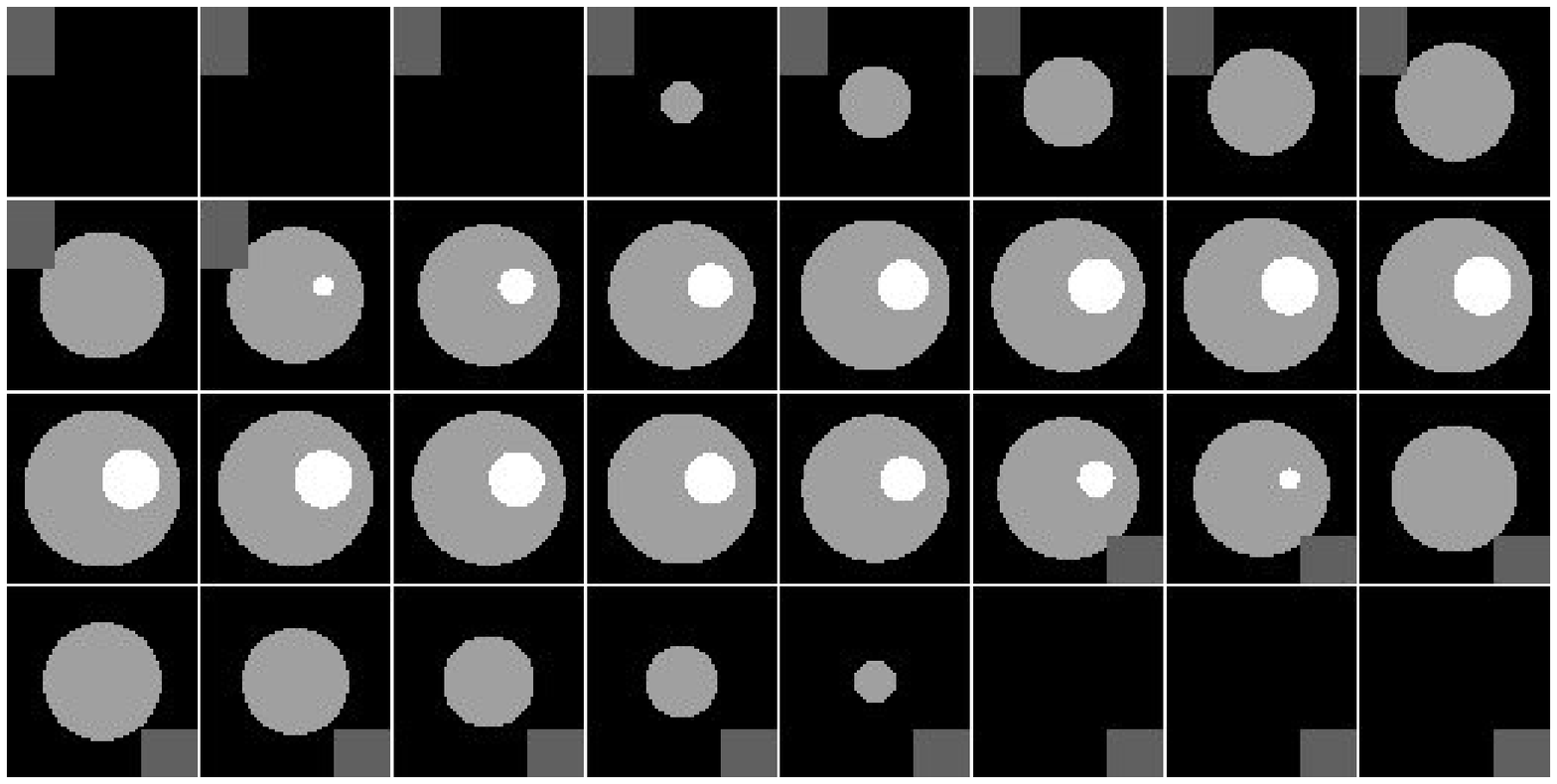}& 
    \includegraphics[trim=3cm 3cm 4cm 6cm,clip,height=3.3cm]{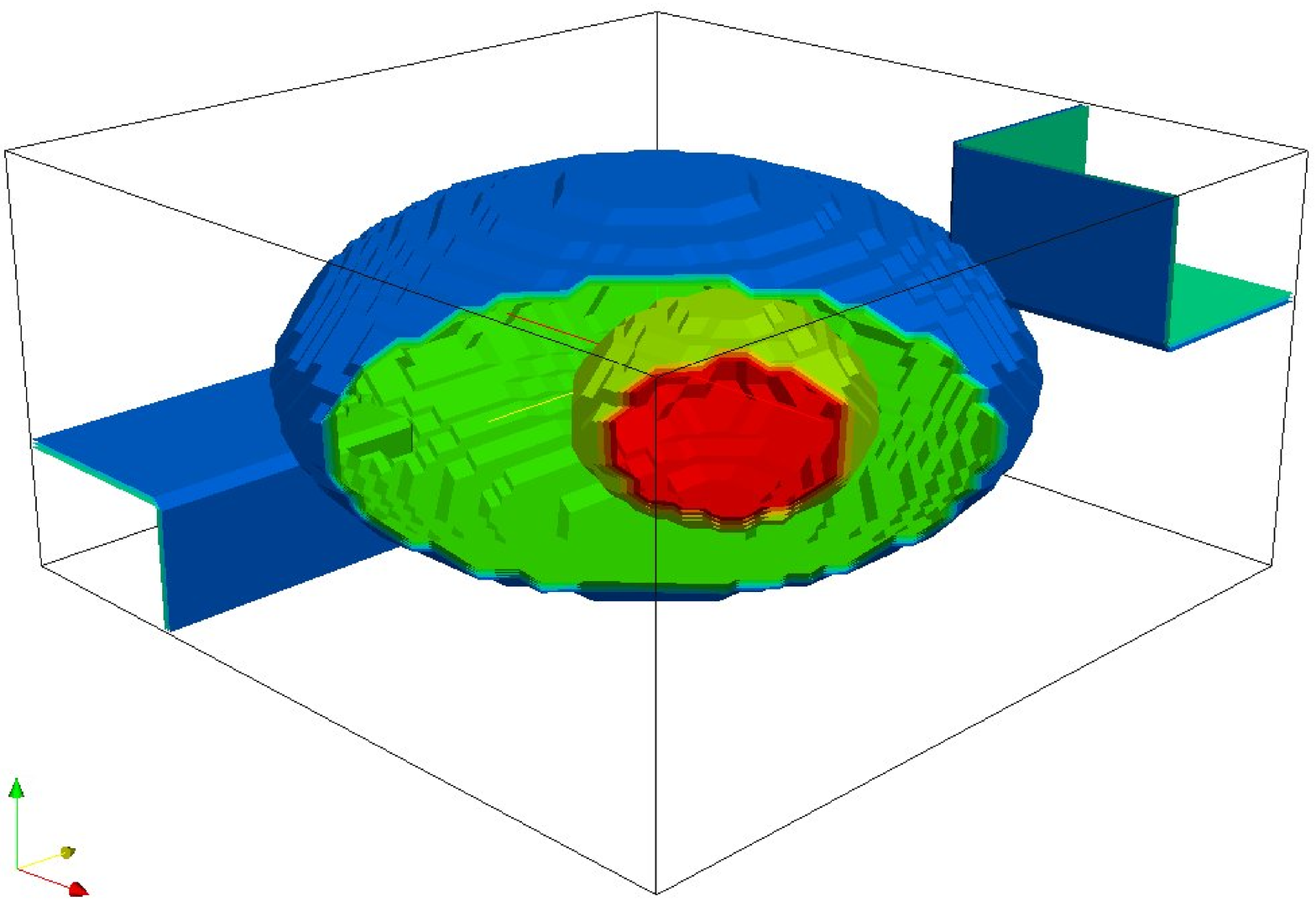} & $u$ \\[3pt]
    \includegraphics[height=3.3cm]{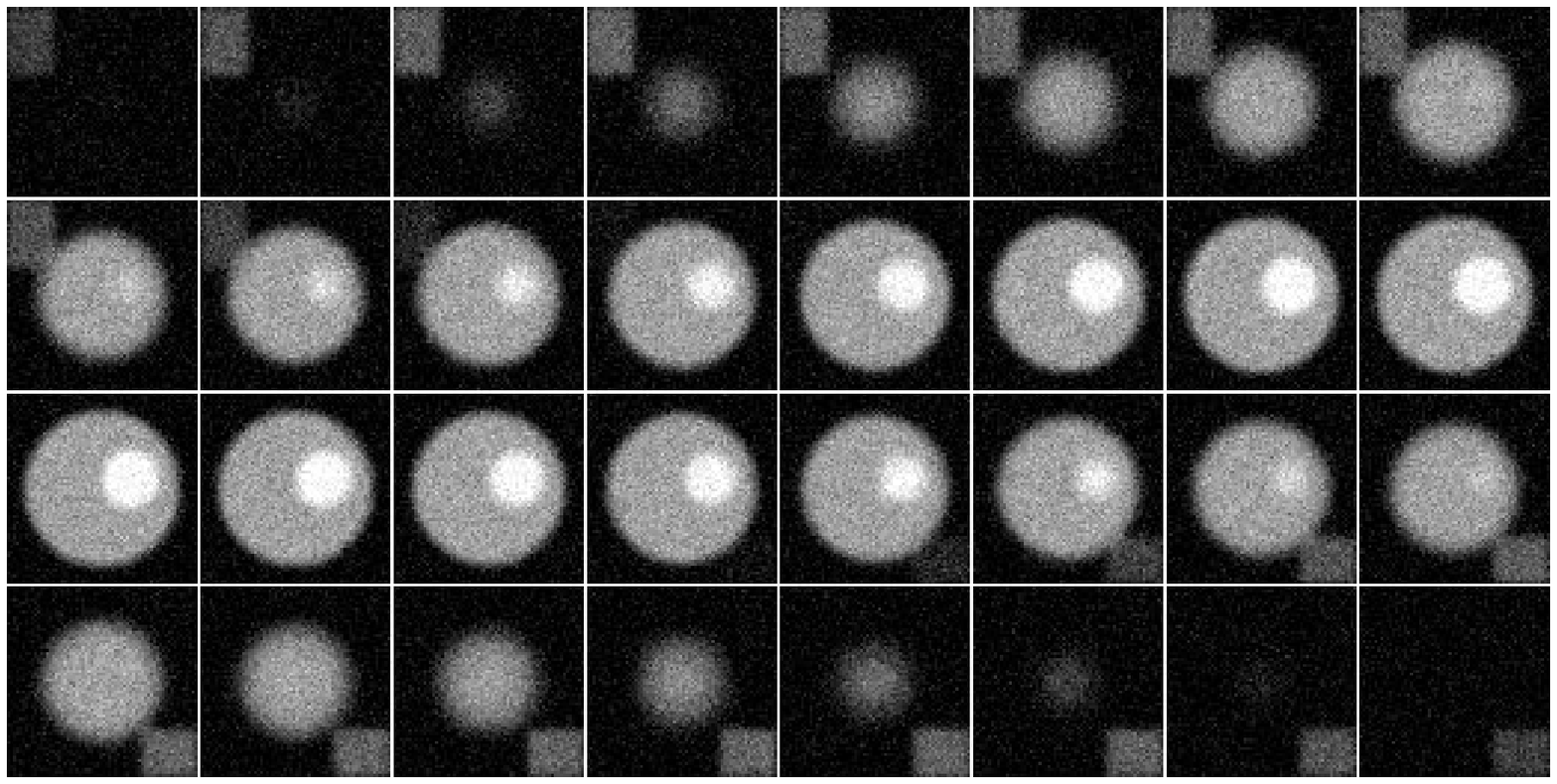}& 
    \includegraphics[trim=3cm 3cm 4cm 6cm,clip,height=3.3cm]{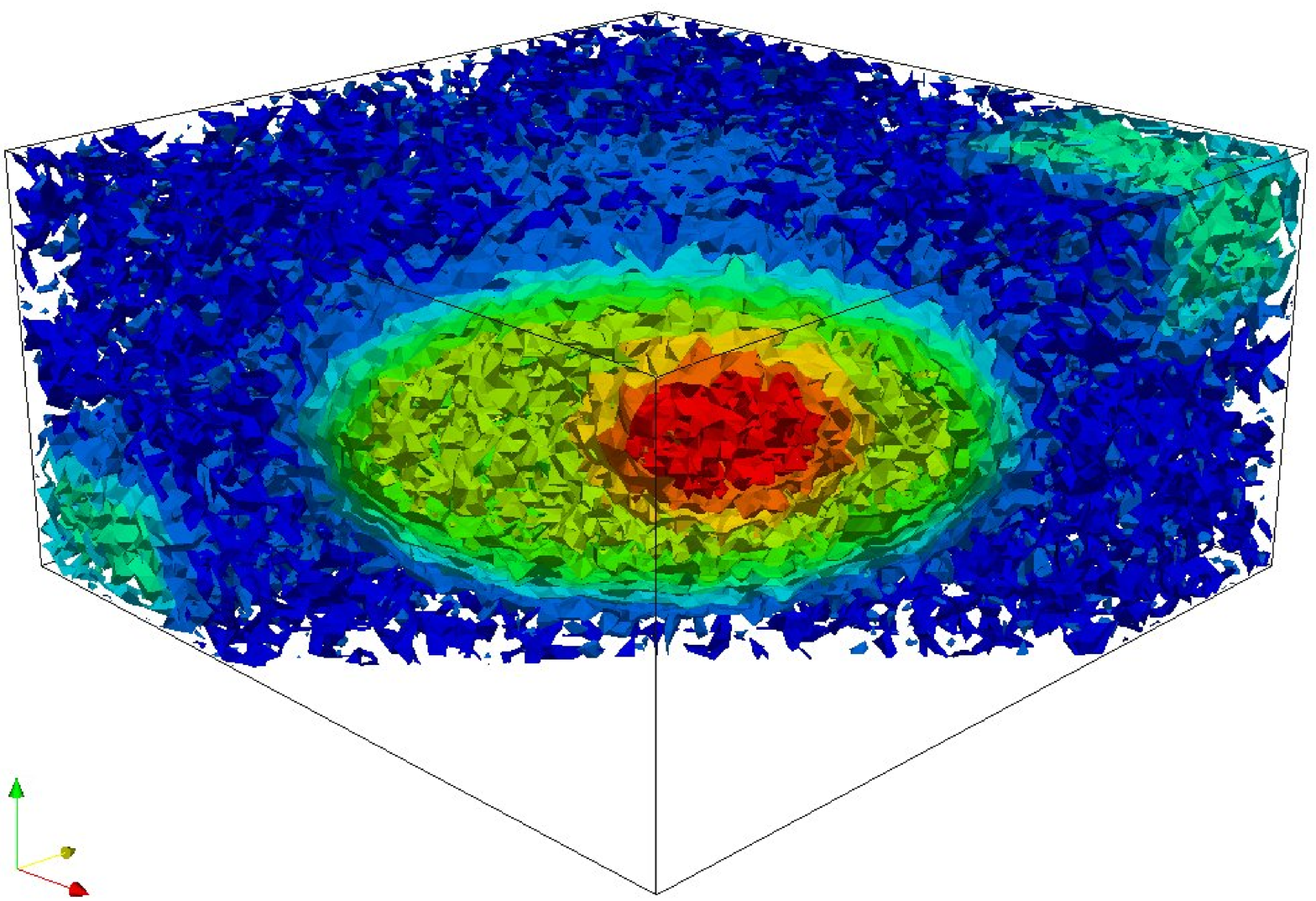} & $f$ \\[3pt]
    \includegraphics[height=3.3cm]{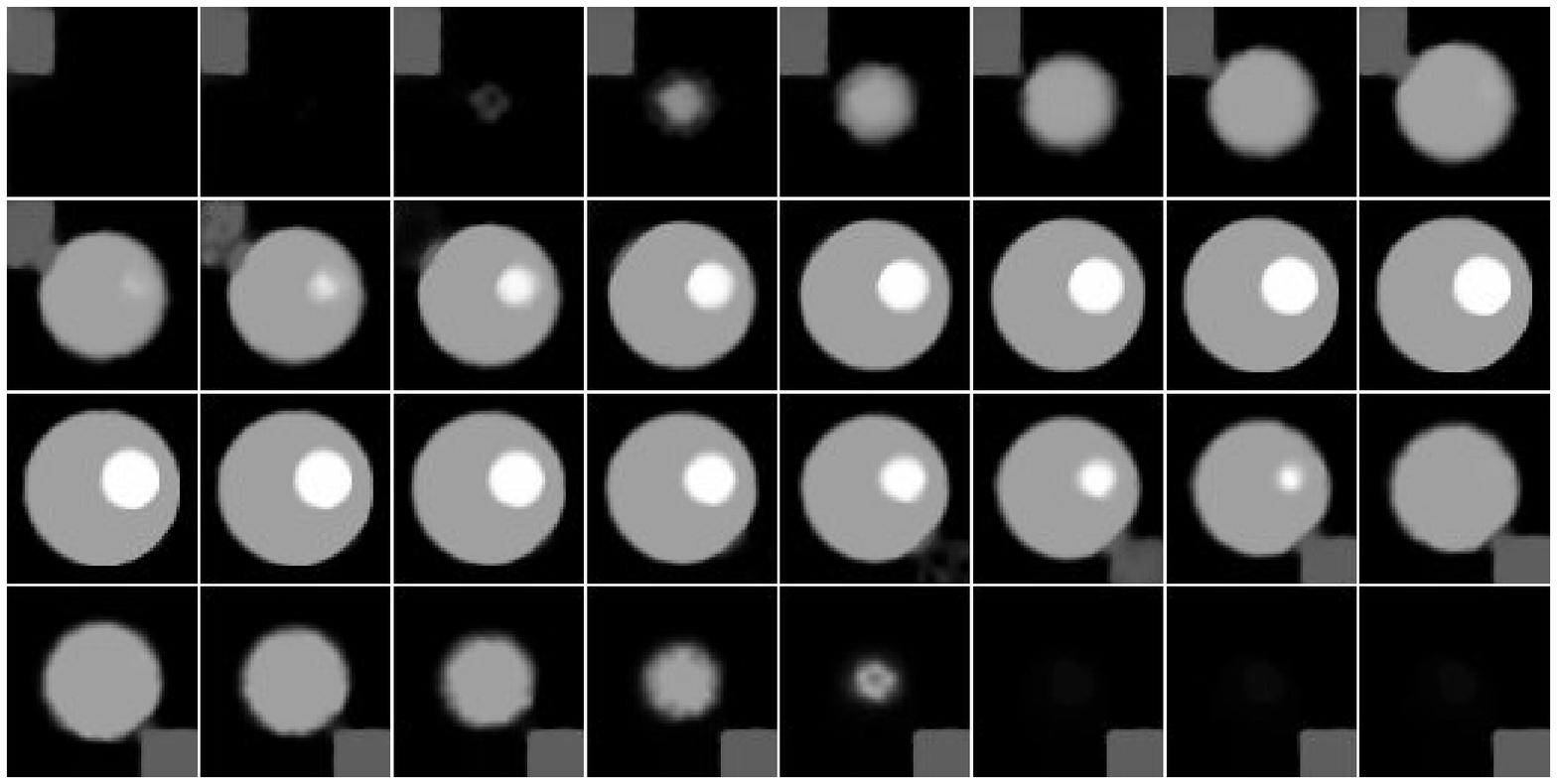}& 
    \includegraphics[trim=3cm 3cm 4cm 6cm,clip,height=3.3cm]{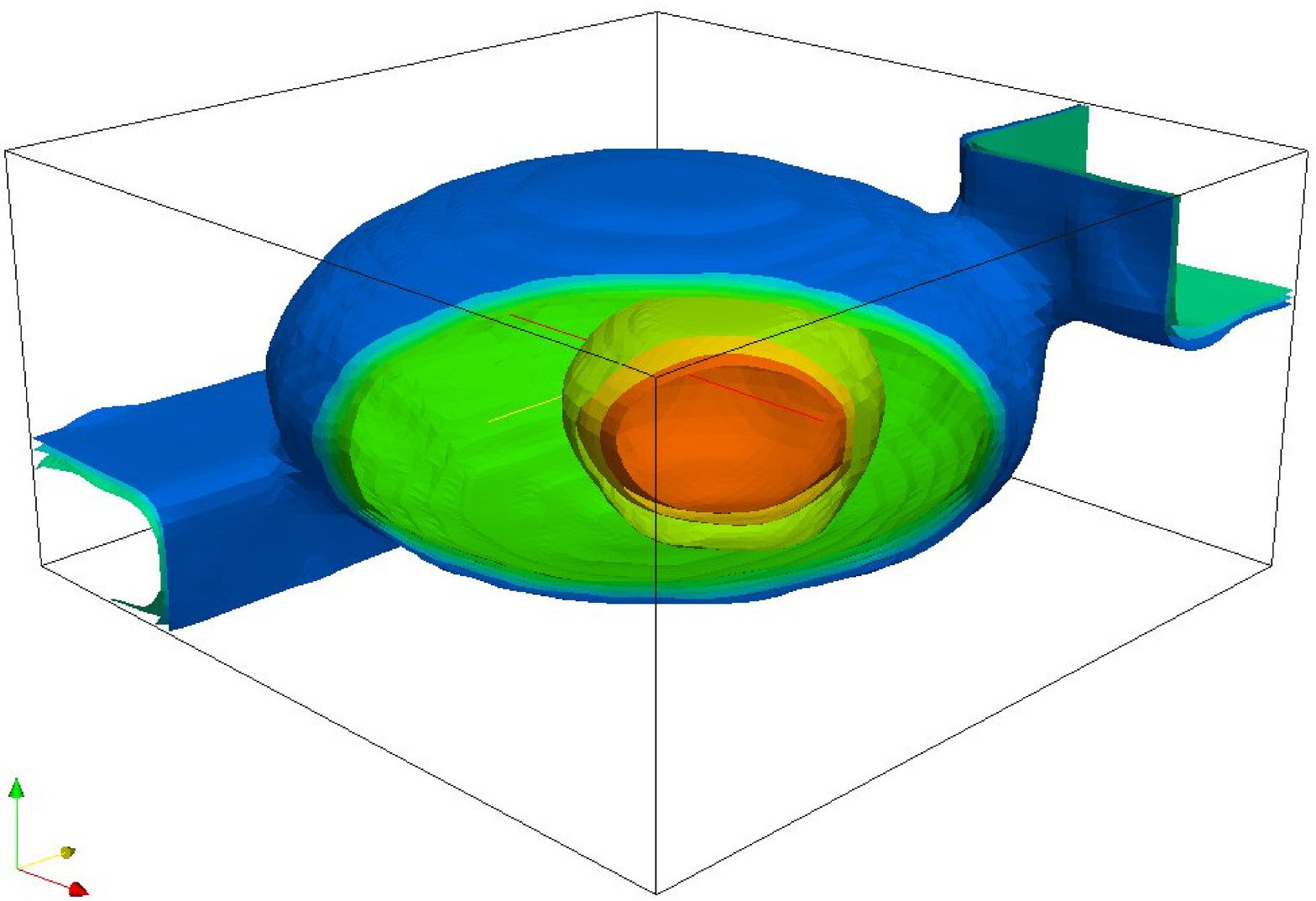} & $u^*$
  \end{tabular}
  \caption{Numerical illustration of three-dimensional total-variation 
    regularization for image deblurring. On the left hand side, you can
    see, respectively, the 2D-slices of the 3D-dataset, while on the
    right hand side, some isosurfaces of a cut of the data is depicted.
    The rows show, from top to bottom,
    the original artificially created data $u$, which
    has been blurred and disturbed with noise to form the data $f$ and
    the outcome of the
    iterative forward-backward splitting algorithm $u^*$ for the deblurring
    problem with total-variation penalization.}
  \label{fig:tv_3d_deconv_art}
\end{figure}

\begin{figure}
  \centering
  \begin{tabular}{ccc} \\
    \includegraphics[height=3.55cm]{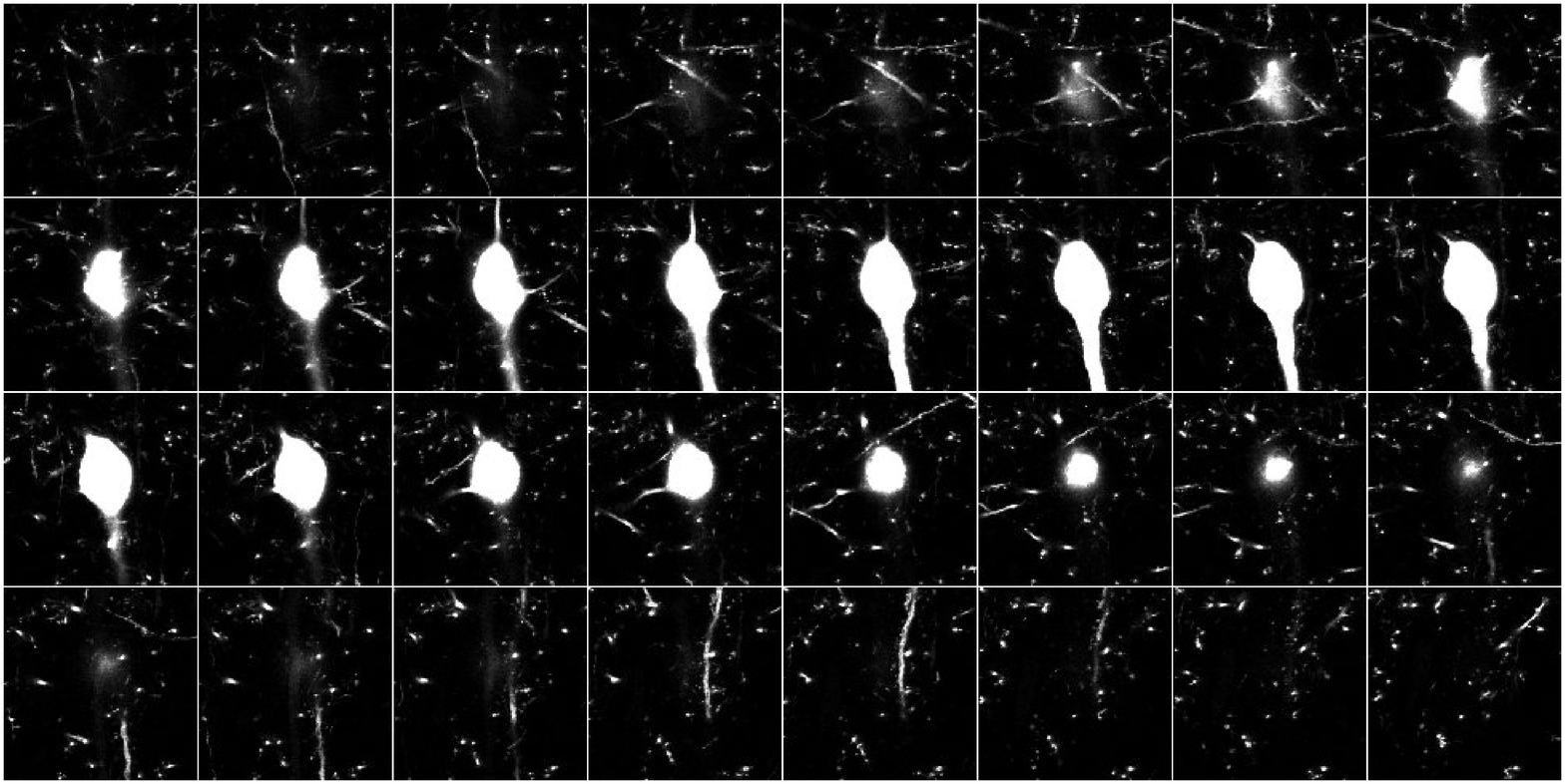}& 
    \includegraphics[trim=3cm 0cm 4.5cm 4cm,clip,height=3.55cm]{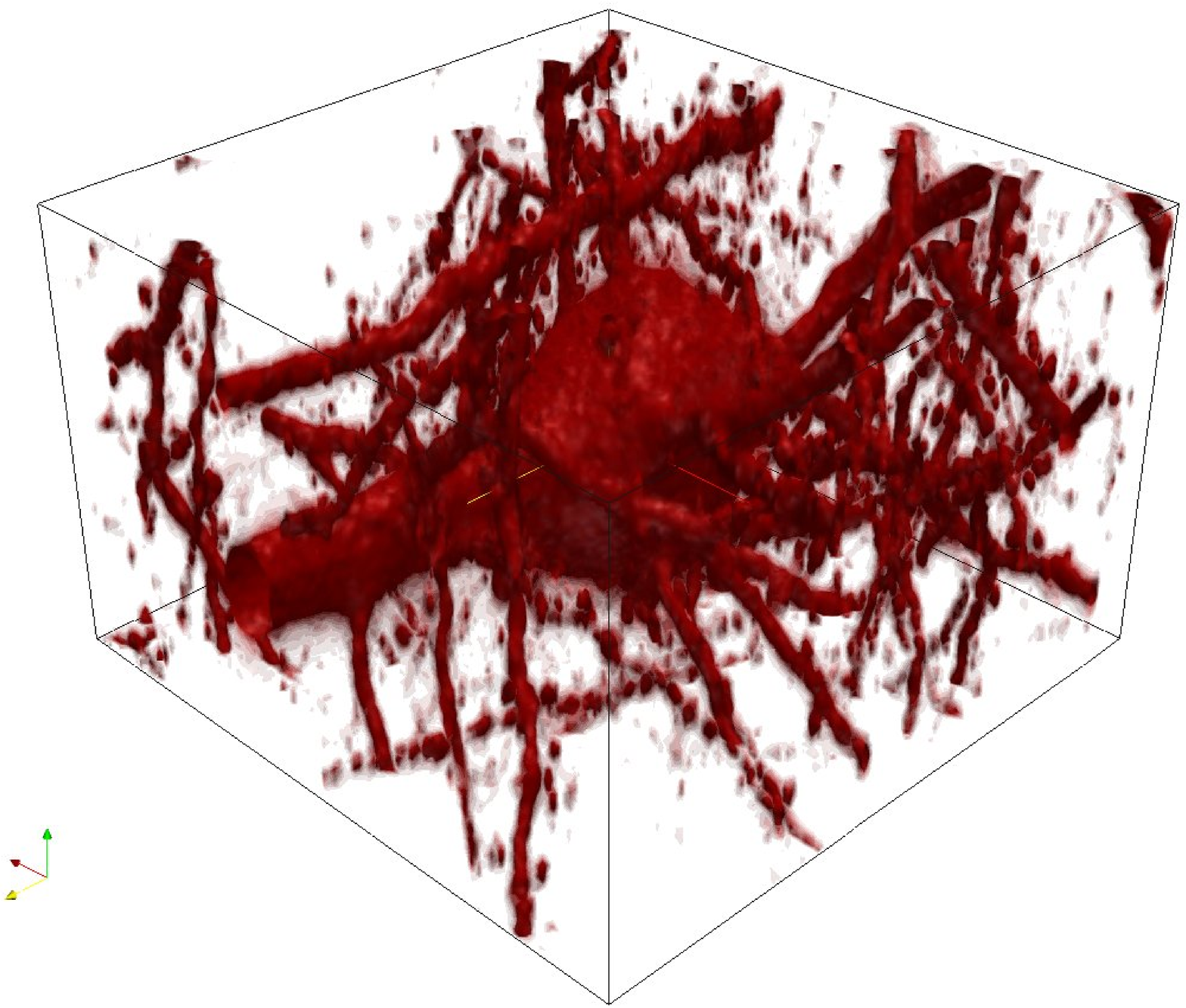} & $u$ \\[3pt]
    \includegraphics[height=3.55cm]{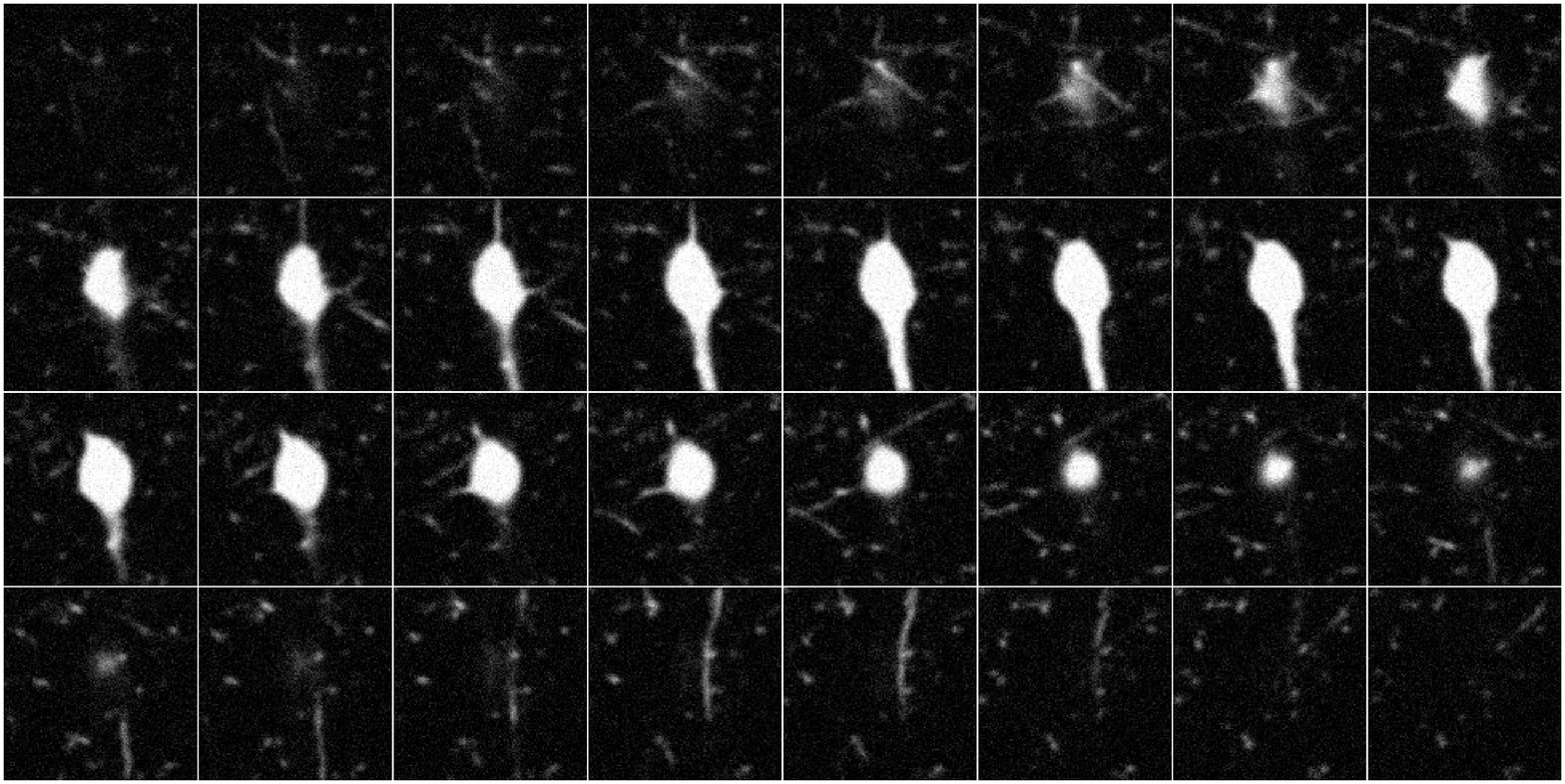}& 
    \includegraphics[trim=3cm 0cm 4.5cm 4cm,clip,height=3.55cm]{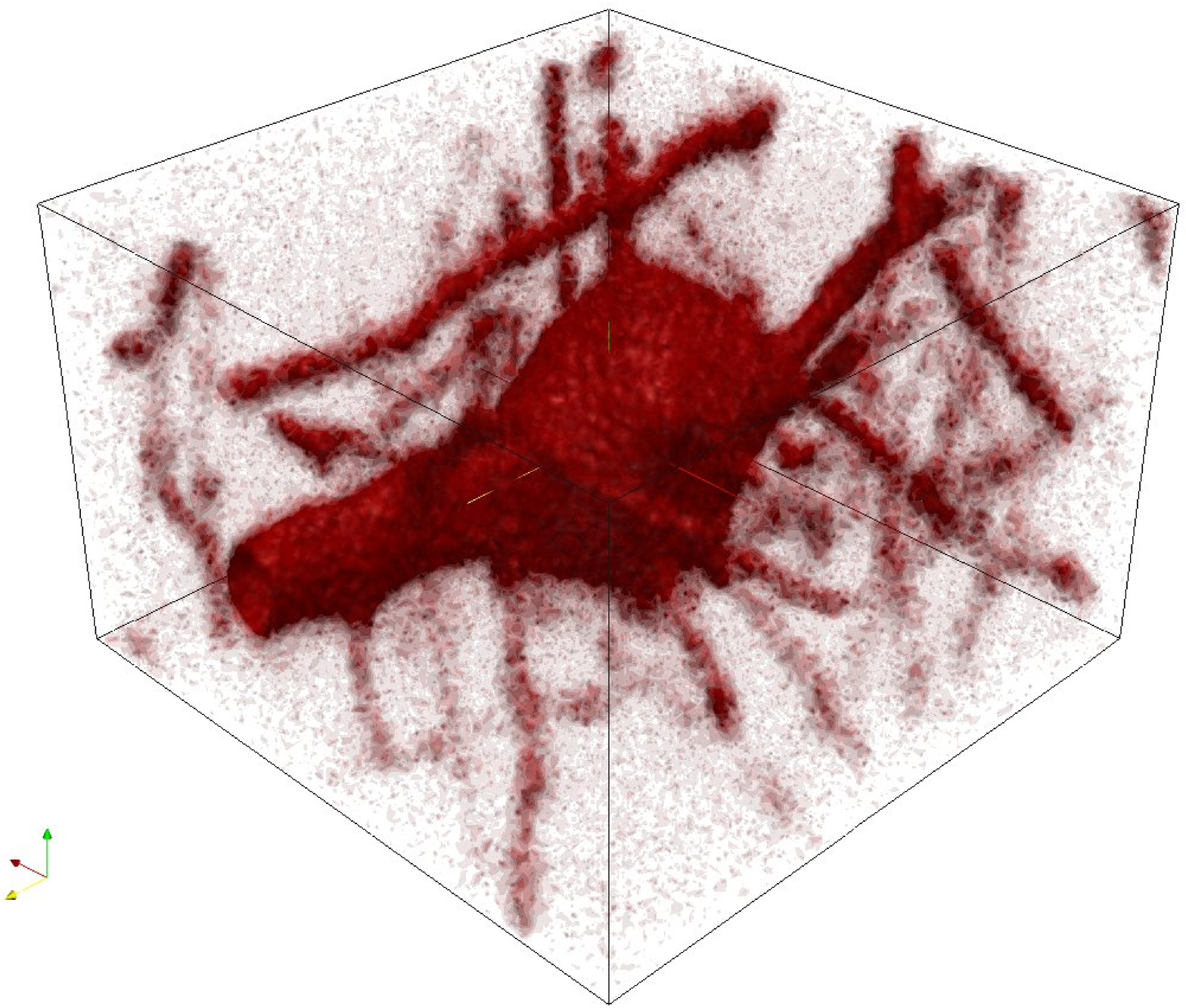} & $f$ \\[3pt]
    \includegraphics[height=3.55cm]{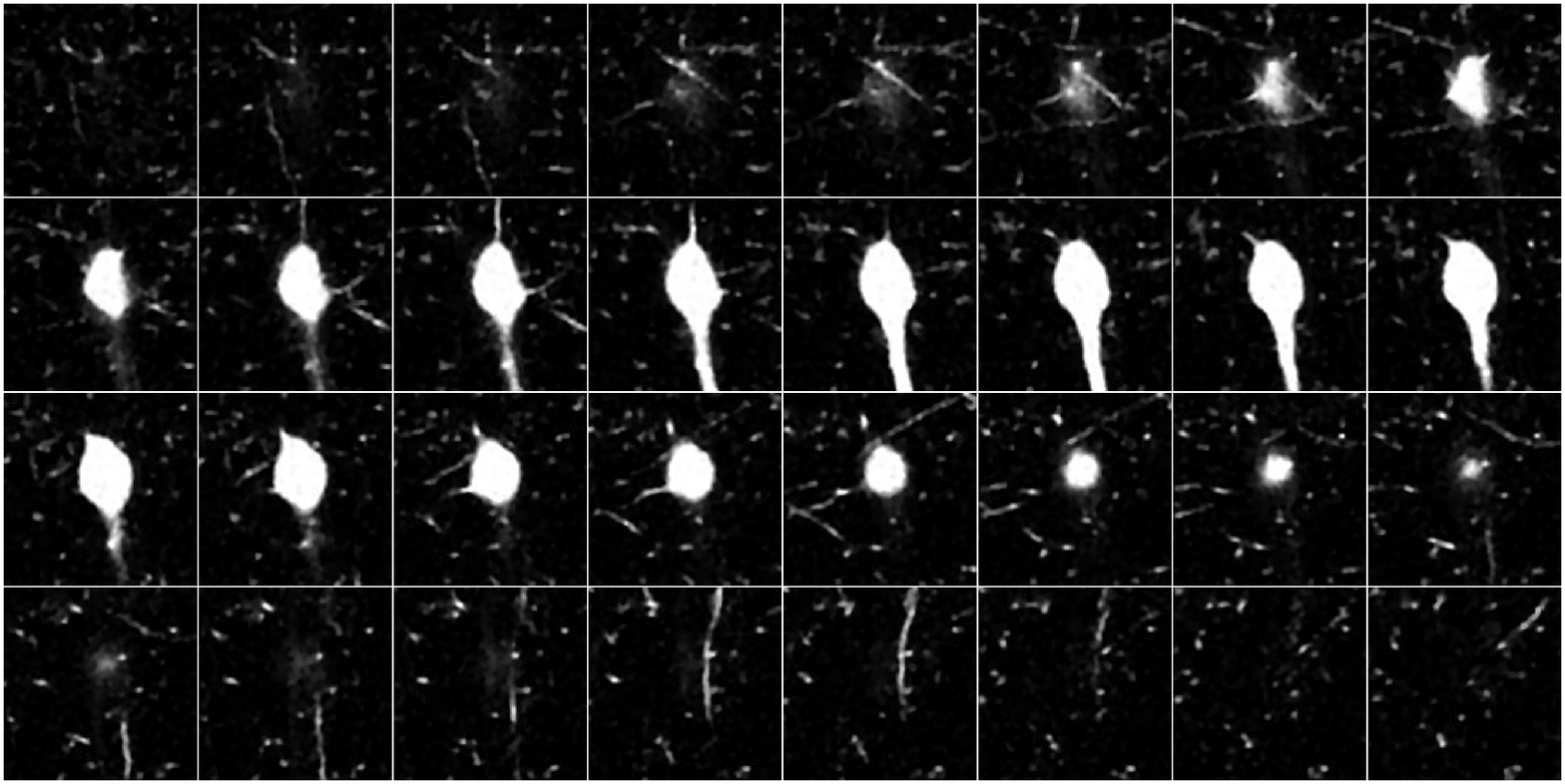}& 
    \includegraphics[trim=3cm 0cm 4.5cm 4cm,clip,height=3.55cm]{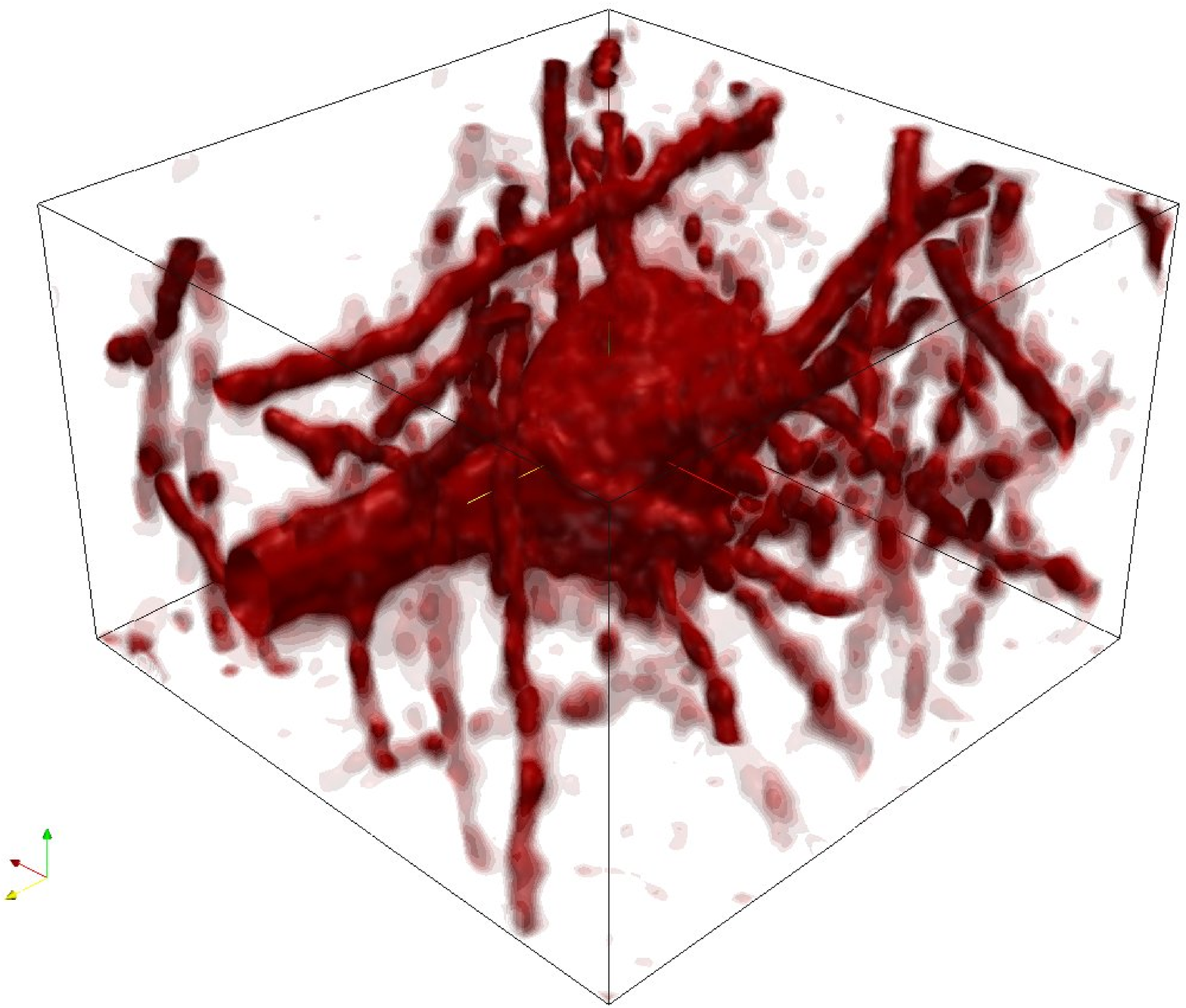} & $u^*$
  \end{tabular}
  \caption{Reconstruction of noisy blurred three-dimensional microscopy data
    showing cortical neurons in transgenic mice.
    Again, slices and isosurface representations of 
    the true image $u$, the noisy data $f$ and the solution of the 
    $\TV$-regularization problem $u^*$ (from top to bottom) are depicted.
    Here, the minimization algorithm is also able to remove the 
    noise artifacts
    from the data. However, a reduction of the contrast and some loss of
    detail can be observed
    as it is typical for total-variation based regularization.
    (Dataset from \texttt{http://152.19.37.82/pages/datasample.php}, 
    see also \cite{feng2000neuronsmice}).}
  \label{fig:tv_3d_deconv_microscopy}
\end{figure}

\section{Summary and Conclusions}
\label{sec:conclusions}

The aim of this paper was to show that there is a meaningful 
generalization of the
forward-backward splitting algorithm to general Banach spaces. The
main idea was to write the forward-backward step as a minimization
problem~\eqref{eq:forward_backward_splitting_hilbert} and to 
generalize this problem to~\eqref{eq:min_prob_aux} which defines
the iteration.
Convergence of this procedure was achieved by proving a descent
rate of $n^{1-p}$ of the functional distance on the hand and utilizing
notions of convexity of the non-smooth functional $\Phi$ to establish
norm-convergence of rate $n^{(1-p)/q}$. This rate is, however, 
rather slow in comparison to, e.g.~linear convergence. But, we
have convergence nevertheless, and 
in order to prove that the general procedure converges, it suffices 
to look at the functional for which the backward-step is performed.

These abstract results were applied to the concrete setting of Tikhonov 
functionals in Banach space. The forward-backward splitting 
algorithm was applied to the computational minimization of  
functionals of Tikhonov-type with semi-norm regularization and, 
using Bregman-Taylor-distance 
estimates, convergence was proven provided the linear operator
has a continuous inverse on the space where the semi-norm vanishes.
In particular, convergence rates translated to the
convexity of the regularizing semi-norm as well as to the
smoothness and convexity of the underlying data space, the latter
originating from the situation that semi-norms are in general invariant
on whole subspaces. 

As the examples showed, the results are applicable for 
Tikhonov functionals considered in practice.
In particular, the algorithm can be used to deal with sparsity
constraints in Banach space and to derive a convergent
generalization of the popular iterative soft-thresholding procedure
by Daubechies, Defrise and De Mol to Banach spaces. The resulting 
algorithm shares many properties: It is easy to implement, 
produces sparse iterates, but also converges very slowly, what was to 
expect since its prototype in Hilbert space also admits very slow 
convergence. Thus, there is the need to accelerate the procedure by, 
for example, utilizing better step-size rules or higher-order methods.
Finally, the method also works for image restoration problems of 
dimension three which can, in general, not be solved in 
$\LPspace{2}{\Omega}$ anymore.  Although the theory does not yield
estimates  for the convergence rate, we are still able to obtain strong 
convergence in some $\LPspace{p}{\Omega}$.

\bibliography{literature,paper}
\bibliographystyle{abbrv}

\end{document}